\documentclass[12pt,reqno]{amsart}

\usepackage[usenames]{color} 
\let\dotlessi\i
\usepackage{pbruillard_commands}

\newcommand{\mbZ}{\mathbb{Z}}
\newcommand{\mbC}{\mathbb{C}}
\newcommand{\mbK}{\mathbb{K}}
\newcommand{\Z}{\mathbb{Z}}
\newcommand{\mbQ}{\mathbb{Q}}
\newcommand{\one}{\mathbf{1}}
\newcommand{\ot}{\otimes}
\newcommand{\inv}{^{-1}}
\newcommand{\ol}{\overline}
\newcommand{\hs}{\hat{\sigma}}
\newcommand{\id}{\operatorname{id}}
\newcommand{\FSexp}{\operatorname{FSexp}}
\newcommand{\GalQ}[1]{\Gal(\mbbQ_{#1}/\mbbQ)}
\newcommand{\dimC}{\dim{\mcC}}

\calclayout

\newcommand{\lcm}{\operatorname{lcm}}

\renewcommand\o{\otimes}
\title[Classification of Weakly Integral Modular Categories]{On the Classification of Weakly Integral Modular Categories}
\date{\today}
\author[P. Bruillard]{Paul Bruillard}
\email{pjb2357@gmail.com}
\address{Pacific Northwest National Laboratory, 902 Battelle Boulevard,
Richland, WA U.S.A}

\author[C. Galindo]{C\'{e}sar Galindo}
\email{cn.galindo1116@uniandes.edu.co}
\address{Departamento de Matem\'aticas, Universidad de los Andes, Bogot\'a, Colombia.}

\author[S.-H. Ng]{Siu-Hung Ng}
\email{rng@math.lsu.edu}
\address{Department of Mathematics, Louisiana State University, Baton Rouge, LA
    U.S.A.}
\author[J. Plavnik]{Julia Y. Plavnik}
\email{plavnik@famaf.unc.edu.ar}
\address{Facultad de Matem\'{a}tica, Astronom\'{\dotlessi}a y F\'{\dotlessi}sica, Universidad Nacional de C\'{o}rdoba, CIEM CONICET, (5000) Ciudad Universitaria, C\'{o}rdoba, Argentina}
\author[E. Rowell]{Eric C. Rowell}
\email{rowell@math.tamu.edu}
\address{Department of Mathematics,
    Texas A\&M University,
    College Station, TX
    U.S.A.}
\author[Z. Wang]{Zhenghan Wang}
\email{zhenghwa@microsoft.com}
\address{Microsoft Research Station Q and Department of Mathematics,
    University of California,
    Santa Barbara, CA
    U.S.A.}
\begin{document}
\begin{abstract}
We classify all modular categories of dimension $4m$, where $m$ is an odd square-free integer, and all ranks $6$ and $7$ weakly integral modular categories. This completes the classification of weakly integral modular categories through rank $7$. Our results imply that all integral modular categories of rank at most $7$ are pointed (that is, every simple object has dimension $1$). All strictly weakly integral (weakly integral but non-integral) modular categories of ranks $6$ and $7$ have dimension $4m$, with $m$ an odd square free integer, so their classification is an application of our main result.  The classification of rank $7$ integral modular categories is facilitated by an analysis of two actions on modular categories: the Galois group of the field generated by the entries of the $S$-matrix and the group of isomorphism classes of invertible simple objects.  The interplay of these two actions is of independent interest, and we derive some valuable arithmetic consequences from their actions.
\end{abstract}
\thanks{The results obtained in this paper were mostly obtained while all 6 authors were at the American Institute of Mathematics, participating in a SQuaRE.  We would like to thank that institution for their hospitality and encouragement. C. Galindo was partially supported by the FAPA funds from vicerrectoria de investigaciones de la Universidad de los Andes, S.-H. Ng  by NSF DMS-1303253 and DMS-1501179, J.Plavnik by CONICET, ANPCyT and Secyt-UNC, E. Rowell by NSF grant DMS-1108725, and Z. Wang by NSF grant DMS-1108736.  PNNL-SA-110195.}
\maketitle

\section{Introduction}
A modular category over $\mathbb{C}$ is a non-degenerate braided spherical fusion category over $\mathbb{C}$.  In this article, we are interested in extending the classification of modular categories over $\mathbb{C}$ of low rank, under some additional constraints on the dimensions of simple objects.
A fusion category $\mcC$ is called \textbf{integral} if the Frobenius-Perron dimensions $\FPdim(X)$ of simple objects $X$ are integers, whereas $\mcC$ is \textbf{weakly integral} if $\FPdim(X)^2\in\mbZ$ for all simple objects $X$.
 We completely classify weakly integral categories when the number of simple objects (up to isomorphism) is $6$ or $7$.  This is facilitated by a more general theorem in which we classify modular categories of dimension $4m$, where $m$ is an odd square-free integer.

Broad classes of integral modular categories arise from finite groups, for
instance the representation categories $\Rep(D^\omega G)$ of twisted Drinfeld doubles $D^\omega G$ of finite groups.  Known constructions of weakly integral modular categories that are not integral (called \textbf{strictly weakly integral}) are somewhat fewer.  Common examples are the modular categories $SO(N)_2$ associated with quantum groups $U_q\mathfrak{so}_N$ at $q=e^{\pi i/(2N)}$, for $N$ odd and $q=e^{\pi i/N}$, for $N$ even (see e.g. \cite{NR}), Drinfeld centers of weakly integral modular categories and $\mbbZ_2$-equivariantizations (see \cite{GNN} and \cite{DGNO1}) of the $G$-crossed extensions of Tambara-Yamagami categories \cite{TY} associated to a group $A$ of odd order.  Indeed, by the proof of Theorem \ref{t:4m}, any strictly weakly integral modular category of dimension $4k$ with $k>1$ an odd
square-free integer and with exactly two isomorphism classes of invertible objects must be of the form $TY(\Z_k,\chi,\nu)^{\mbbZ_2}$, as a braided fusion category.

The $\mbbZ_2$-equivariantization mentioned above is a special case of a general method, called \textbf{gauging} \cite{BBCW, CGPW}, for producing new weakly integral modular categories.  Given a weakly integral modular category $\mathcal{C}$ with a categorical action (or symmetry) of a finite group $G$, then gauging the $G$ symmetry of $\mathcal{C}$ as defined in \cite{BBCW, CGPW} results in a weakly integral modular category denoted $\mathcal{C}//G$.  An interesting weakly integral modular category is obtained by gauging the $S_3$ symmetry of $SO(8)_1$ in \cite{CGPW} described as Example J in \cite{BBCW}.  While representations of the braid groups from weakly integral modular categories are all conjectured to have finite images \cite{NR}, they are still useful for topological quantum computation \cite{NR}.  Actually, weakly integral modular categories are potentially more easily realized in nature than non-weakly integral categories. Moreover, they can be made universal for quantum computation in some cases 
such as the metaplectic modular categories \cite{CW}.

The following classes of weakly integral modular categories appear in our
classification:
\begin{enumerate}
  \item A \textbf{pointed} modular category is a modular category in which all simple objects are \textbf{invertible}, that is, $X\ot X^*\cong \1$.  Such categories have
    the same fusion rules as $\Rep(A)$ where $A$ is a finite abelian
    group. A \textbf{cyclic} modular category is a pointed modular
    category with the same fusion rules as $\Rep(\mbbZ_{n})$.
  \item An \textbf{Ising} modular category $\mcI$ is a non-pointed modular
    category with $\dim\mcI=4$ (see \cite[Appendix B]{DGNO2} where such
    categories are classified).  Such a category has rank $3$ with two simple
    classes of dimension $1$ and one of dimension $\sqrt{2}$.
  \item A \textbf{metaplectic} modular category $\mcC$ is any modular category
    with the same fusion rules as
    $SO(N)_2$ for $N$ odd (see e.g. see \cite{NHW}). These rank $\frac{N+7}{2}$, $4N$-dimensional categories have two
    $1$-dimensional objects and two objects of dimension $\sqrt{N}$, while the
    remaining $\frac{N-1}{2}$ objects have dimension $2$.  For example,
    $TY(\mbbZ_N,\chi,\nu)^{\mbbZ_2}$ for $N$ odd is a metaplectic modular category (see \cite{GNN}).
\end{enumerate}
\begin{rmk}
There are exactly eight unitary Ising categories \cite{kitaev}.  {\it The}
Ising theory refers to the one with central charge $c=\frac{1}{2}$ and the
topological twist of the non-pointed simple object is
$\theta=e^{\frac{2\pi i}{16}}$.
\end{rmk}

Our results are summarized in:
\begin{theorem}
Suppose that $\mcC$ is a weakly integral modular category either of rank $\leq
7$ or with $\FPdim(\mcC)=4m$, where $m$ is an odd square-free integer.  Then
$\mcC$ is equivalent to a (Deligne) product of the following: pointed
categories, Ising categories and metaplectic categories.
\end{theorem}

Our approach involves a number of well-established classification techniques as
well as some new ones.  A central theme will be the application of group actions
on modular categories $\mcC$.  In particular, two groups act on equivalence
classes $\Irr\(\mcC\)$ of simple objects: the Galois group $\Gal(\mcC)$ associated
to the splitting field of the Grothendieck ring $K_0(\mcC)$ and the group
$G(\mcC)$ of invertible objects which is nontrivial for strictly weakly integral
modular categories.  We study the relationship between these
two group actions as well as their structures, which yield some useful
constraints on the dimension of the classified categories.

\section{Preliminaries}
In this section, we set the notation and recall some useful results from the literature.
Of particular importance are the groups $G(\mcC)$ of isomorphism classes of invertible objects and $\Gal(\mcC)$ the Galois group of the number field generated by the entries of the $S$-matrix.
\subsection{Modular Data}
Let $\mcC$ be a modular category (see \cite{BK} for a full list of axioms).
We denote by $\Irr\(\mcC\)$ the set of isomorphism class of simple objects of $\mcC$.  A representative of $i \in \Irr\(\mcC\)$ will be
denoted by $V_i$ or $X_i$, and the isomorphism class of the unit object $\1$ will be labelled  $0\in\Irr\(\mcC\)$. We denote by $S_{ij}$ the $(i,j)$-entry of the (unnormalized)
$S$-matrix for
$i,j \in \Irr\(\mcC\)$ and by $d_i=\dim(V_i)$ the pivotal dimension of the object $V_i$. The categorical dimension $\dim (\mcC) = \sum_i d_i^2$ is
independent of the pivotal structure of $\mcC$ and  the fusion coefficients are
given by $N_{ij}^k:=\dim\Hom(V_i\otimes V_j,V_k)$ and the twists are denoted
$\theta_i:=\theta_{V_i}$. The (unnormalized) $T$-matrix of $\mcC$ is defined by
$T_{ij} = \delta_{ij} \theta_i$ and has finite order (cf. \cite{BK}).  The order
of $T$ is equal to the Frobenius-Schur exponent of $\mcC$ (see \cite{NS} for
more details) denoted $\FSexp(\mcC)$.  The \textbf{Gauss sums}
$p_{\pm}:=\sum_{i\in\Irr\(\mcC\)} d_i^2\theta_i^{\pm 1}$ satisfy $p_+p_-=\dim\mcC$
and $\left(\frac{p_+}{p_-}\right)^{2N}=1$.  The pair $(S,T)$ is called the
\textbf{modular data} of $\mcC$, and it satisfies the following equations (see
\cite[Sect. 2]{BNRW} or \cite[3.1]{BK}).
\begin{enumerate}
 \item \textbf{Twist equation}: \begin{equation}\label{twist equation} p_+S_{ij}=\theta_i\theta_j\sum_k S_{ki}S_{kj}\theta_k. \end{equation}
 \item \textbf{Balancing equation}: \begin{equation}\label{balancing equation} S_{ij}\theta_i\theta_j=\sum_k N_{i^*j}^kd_k\theta_k.\end{equation}
 \item \textbf{Orthogonality}: \begin{equation}\label{orthogonality} SS^{\dag}=\dim(\mcC)\, I,\end{equation} where $\dag$ is the conjugate-transpose operation and $I$ denotes the identity matrix.
\end{enumerate}

The twist equation comes from the relation $(ST)^3=p_{+}S^2$.  While $(S,T)$ give rise to a projective representation of the modular group $\SL(2,\mbbZ)$, one may normalize $s=S/x$, $t=T/y$ so that $(s,t)$ gives a linear representation of $\SL(2,\mbbZ)$.  We call a pair $(s,t)$ obtained in this way a \textbf{normalized modular data} for $\mcC$.

The \textbf{Grothendieck ring} $K_0(\mcC)$ of $\mcC$ has $\mathbb{N}$-basis $\Irr(\mcC)$ with addition and multiplication induces from $\oplus$ and $\otimes$.
Defining $(N_i)_{k,j}=N_{i,j}^k$ we obtain a representation of the Grothendieck
ring via $V_i\rightarrow N_i$ by extending linearly.  The
Frobenius-Perron dimension $\FPdim(X)$ is defined as the largest eigenvalue of
$N_X$, and $\FPdim(\mcC)=\sum_{i\in\Irr\(\mcC\)}\FPdim(V_i)^2$ (see \cite{ENO1}).

\subsection{$G$-grading}

A fusion category $\mcC$ is \textbf{$G$-graded} ($G$ a finite group) if $\mcC=\bigoplus_{g\in\G} \mcC_g$ as an abelian category and $\mcC_g\otimes\mcC_h\subset\mcC_{gh}$. If each $\mcC_g$ is non-empty, the grading is called \textbf{faithful}. It was proved in \cite[Theorem 3.5]{GN} that any fusion category $\mcC$ is naturally graded by a group $U(\mcC)$, called the \textbf{universal grading group} of $\mcC$, and the adjoint subcategory $\mcC_{ad}$ (generated by all subobjects of $X^*\ot X$, for all $X$) is the trivial component of this grading. Moreover, any other faithful grading of $\mcC$ arises from a quotient of $U(\mcC)$ \cite[Corollary 3.7]{GN}. For a modular category, the universal grading group $U(\mcC)$ is isomorphic to the group $G(\mcC)$ of isomorphism classes of invertible simple objects of $\mcC$ \cite[
Theorem 6.2]{GN}. This group will play an important role in later sections. The fusion subcategory generated by the group of invertible objects is the maximal pointed subcategory of $\mcC$ and it is denoted $\mcC_{pt}$.
We will say $\mcC$ is \textbf{strictly weakly integral} if it is weakly integral and has a simple object of non-integral dimension.  A strictly weakly integral fusion category is faithfully graded by an elementary abelian $2$-group \cite[Theorem 3.10]{GN}. Indeed, as each simple object has dimension $\sqrt{k}$, for some $k\in\mbZ$, one may partition the simple objects into finitely many non-empty sets of the form $A_n:=\{X:d_X\in\sqrt{n}\mbZ\}$, where $n$ is square-free, and this partition induces a faithful grading by an elementary $2$-group. The trivial component $\mcC_e$ with respect to this grading is the subcategory $\mcC_{int}$ generated by the simple objects of $\mcC$ of integral dimension.

We should emphasize that $U(\mcC)\cong G(\mcC)$ for modular categories, and the (pointed) subcategory generated by $G(\mcC)$ is denoted $\mcC_{pt}$.

\subsection{M\"uger Center}
Two objects $X$ and $Y$ of a braided fusion category $\mcC$ (with braiding $c$) are said to \emph{centralize each other} if
$c_{Y,X} c_{X,Y} = \id_{X\otimes Y}$.  M\"uger has shown that $X$ and $Y$ centralize each other if and only if $S_{XY}=d_Xd_Y$ \cite[Prop. 2.5]{M2}.  The \textbf{centralizer} $\mcD'$ of a fusion subcategory $\mcD \subseteq \mcC$ is defined to be the full
subcategory of objects of $\mcC$ that centralize every object of $\mcD$, that is

\begin{equation*} \mcD'=\{ X\in \mcC \,\mid\, c_{Y,X} c_{X,Y} = \id_{X\otimes Y}, \text{for all }  Y\in\mcD \}.
\end{equation*}

The \textbf{M$\ddot{u}$ger center} $Z_2(\mcC)$ of a braided fusion category $\mcC$ is the centralizer of $\mcC$, that is $Z_2(\mcC) = \mcC'$, which is a symmetric fusion subcategory of $\mcC$. A braided fusion category $\mcC$ is called \textbf{symmetric} when $Z_2(\mcC) = \mcC$. A braided fusion category is \textbf{non-degenerate} if  $Z_2(\mcC) \cong \Vec$.  For premodular (spherical braided fusion) categories, non-degeneracy is equivalent to modularity, i.e. the invertibility of the $S$-matrix (\cite{Brug}).

If $\mcD\subset \mcC$ are both modular, then the centralizer $\mcD^\prime$ is also modular \cite{M2}.  In particular, one has $\mcC\cong \mcD\boxtimes\mcD^\prime$ where $\boxtimes$ denotes the \textbf{Deligne product}, see \cite{Del}.  In general $\mathcal{A}\boxtimes\mathcal{B}$ has rank $ab$ where the ranks of $\mathcal{A}$ and $\mathcal{B}$ are $a$ and $b$, respectively.
\subsection{Galois Group Action}
Let $\mcC$ be a modular category with modular data $(S, T)$. We define the
\textbf{Galois group} $\Gal(\mcC)$ of $\mcC$ as
$\Gal(\mcC):=\Gal(\mathbb{K}_{\mcC}/\mbbQ)$, where $\mathbb{K}_{\mcC} =
\mathbb{Q}(S_{ij}\mid i, j \in \Irr\(\mcC\))$. Notice that $\mathbb{K}_{\mcC}$ is the splitting field of the set of characteristic polynomials of the fusion matrices $N_i$.  The action of $\Gal(\mcC)$ on the
columns of the $S$-matrix induces a faithful action on $\Irr\(\mcC\)$ (cf.
\cite{BNRW}).  We will denote by $\hs$ the permutation on $\Irr\(\mcC\)$ associated with $\s \in \Gal(\mcC)$.

Let $\zeta_N=e^{\frac{2\pi i}{N}}$, then we have the exact sequence:
$$1\rightarrow \Gal(\mbbQ(\zeta_N)/\mathbb{K}_{\mcC})\rightarrow \Gal(\mbbQ(\zeta_N))\rightarrow \Gal(\mcC)\rightarrow 1,$$
and it is known (\cite{NS2}) that $\Gal(\mbbQ(\zeta_N)/\mathbb{K}_{\mcC})$ is an elementary $2$-group.

It follows from \cite[Lem. 4.9]{BNRW} that if $d_i = \FPdim(V_i)$, for $i \in \Irr\(\mcC\)$, then
$$
\hs(d_i) = d_{\hs(i)}\,.
$$
In particular, $\hs(0) \in G(\mcC)$.

If $(s,t)$ is a normalized modular data for $\mcC$ then Galois Symmetry (cf. \cite{DLN1}) implies: $$\s^2(t_i)=t_{\hs(i)}  \text{ for all } i.$$

\section{$\dim\mcC=4m$, $m$ odd square-free integer}
Using \cite[Prop. 8.23]{ENO1} we may and will assume that all weakly integral modular categories in the following have the canonical positive spherical structure, with respect to which categorical dimensions of simple objects coincide with their Frobenius-Perron dimensions.

\begin{theorem}\label{t:4m}
Suppose that $\mcC$ is a  modular category with $\dim\mcC=4m$, with $m$ an odd square-free integer. Then exactly one of the following is true:
\begin{enumerate}
\item[(a)] $\mcC$ contains an object of dimension $\sqrt{2}$ and $\mcC$ is equivalent to a Deligne product of an Ising modular category and a cyclic modular category or
\item[(b)] $\mcC$ is non-integral and contains no objects of dimension $\sqrt{2}$ and $\mcC$ is equivalent to a Deligne product of a metaplectic category of dimension $4k$ and a cyclic modular category of dimension $n$ where $1\leq n=\frac{m}{k}\in\mbZ$ or
\item[(c)] $\mcC$ is pointed.
\end{enumerate} All equivalences are as premodular fusion categories.
\end{theorem}
\begin{proof}

If $d_i \in \mbZ$ then $d_i=1$ or $2$ since $d_i^2\mid 4m$ (by \cite[Lemma 1.2]{EG}, \cite[Proposition 8.27]{ENO1}, \cite[Proposition 2.11(i)]{ENO2}).

First, we claim that if $4 \mid \dim \mcC_{pt}$ then $\mcC$ is pointed. Under this assumption, with respect to the universal) $U(\mcC)$-grading, each component has odd dimension. In particular, $\dim \mcC_{ad}$ is odd, so $\mcC_{ad}\subset\mcC_{int}$ cannot contain simple objects of even dimension. Therefore, $\mcC_{ad}$ is pointed and hence $\mcC$ is nilpotent.  By \cite[Thm. 1.1]{DGNO2}, $\mcC$ has a unique decomposition into a tensor product of braided fusion categories whose Frobenius-Perron dimensions are the distinct prime power factors of $\dim \mcC=4m$. Since braided fusion categories with prime Frobenius-Perron dimensions are pointed and $4 \mid |G(\mcC)|$, all the factors in this tensor decomposition of $\mcC$ are pointed. Thus, $\mcC$ is pointed.

If $\mcC$ is integral, then $4m = \dim \mcC_{pt} + 4 s$, where $s = |\{i \in
\Irr\(\mcC\)\mid d_i =2\}|$.  Therefore, $4 \mid \dim \mcC_{pt}$ and hence, by the previous paragraph, $\mcC$ is pointed as described in (c).

We may therefore assume that $4 \nmid \dim \mcC_{pt}$ and $d_i \not\in \mbZ$ for
some $i \in \Irr\(\mcC\)$. Then $|U(\mcC)|= \dim \mcC_{pt} = 2^an$, where $a= 0,1$ and $n\mid m$ is odd and square-free. By \cite[Theorem 3.10]{GN} $a\neq 0$, so we must have $a=1$. 

We claim that $\mcC\cong \mcD\boxtimes \mcP$, where $\mcD$ is modular with $\dim\mcD=4k$ and $\dim\mcD_{pt}=2$ and $\mcP$ is a cyclic modular category of dimension $n=m/k$. We prove this by induction on the number of (distinct) primes factors of
$n$.  In the base case $n=1$ ($0$ prime factors) the statement is clear, since $\mcC\cong\mcC\boxtimes\Vec$. Suppose $p$ is a prime factor of $n$. Let $\mcC_{pt}(p)$ be a pointed subcategory of dimension $p$. Observe that $\mcC_{pt}(p)$ is either symmetric or modular. If $\mcC_{pt}(p)\cong\Rep(\mbZ_p)$ is symmetric then it is also Tannakian since $p$ is odd (cf. \cite [Corollary 2.50 (i)]{DGNO1}). The corresponding $\mbZ_p$-de-equivariantization of $\mcC$ is $\mbZ_p$-graded with trivial component $(\mcC_{\mbZ_p})_0$ of dimension $\dim\mcC/p^2$, contradicting $n$ square-free. Therefore $\mcC_{pt}(p)$ is modular and hence $\mcC\cong \mcC_{pt}(p)\boxtimes\mcD$ where $\mcD$ is the modular subcategory which centralizes $\mcC_{pt}(p)$ by  \cite[Theorem 4.2]{M2} (see also \cite[Theorem 3.13]{DGNO1}).  Now 
$|U(\mcD)|=2n/p$ so the claim follows.

We now proceed to classify modular categories $\mcD$ with $\dim\mcD=4k$ and $\dim\mcD_{pt}=2$. If $k=1$, then $\dim \mcD=4$ and $\mcD$ is an Ising modular category by \cite[Appendix B]{DGNO1}. In this case, $\mcC$ is as in (a).

Now, we may  assume $k > 1$. Since $U(\mcD)\cong\mbZ_2 $, $\mcD$ has a $\mbZ_2$-grading $\mcD = \mcD_0 \oplus \mcD_1$ where $\mcD_0 = \mcD_{ad}$ is the trivial component.  For any simple object $X_i$ with $d_i \not\in \mbZ$, $X_i \in \mcD_1$ and $d_i=\alpha_i\sqrt{\ell}$, with $\ell$ square-free and $\alpha_i\in\{1,2\}$, since $\mcD$ is modular and $k$ is square free. The trivial component $\mcD_0$ contains all the simple objects of integral dimension, which consist of two isomorphism classes of dimension $1$ and $(k-1)/2$ classes of dimension $2$.

Let $\s \in \Gal(\mcD)$ which assigns $\sqrt{\ell} \mapsto -\sqrt{\ell}$ and
$\hs(0)=1$. Then the second row of the $S$-matrix of $\mcD$ is given by the
action $\s$ on the first row of $S$. Therefore, the second row of $S$
(corresponding to label $1\in\Irr\(\mcC\)$) is:
 $$(1,1,2,\ldots,2,-\alpha_
1\sqrt{\ell},\ldots,-\alpha_t\sqrt{\ell}).$$
In particular, $1\in\Irr\(\mcC\)$ must be the isomorphism class of a non-trivial invertible object $g$.  We claim that $\theta_g=1$. To see this, first observe that if $Y$ is simple and $\dim(Y)=2$ then $Y\otimes Y^*=\one\oplus g\oplus Y^\prime$, where $Y^\prime$ is a $2$-dimensional simple object. Thus, $N_{Y,Y^*}^g=N_{Y,g}^Y=1$ so that $Y\otimes g\cong Y$, for all $2$-dimensional simple object $Y$.
The balancing equation \eqref{balancing equation} gives:
$$2=S_{g,Y}=\theta_g^{-1}\theta_{Y}^{-1}N_{g^*,Y}^Yd_Y\theta_Y=2\theta_g\inv,$$ so that $\theta_g=1$, as claimed.

For $i \in \Irr\(\mcD_{1}\)$, $d_i = \a_i \sqrt{\ell}$. We use the balancing equation \eqref{balancing equation} again:
$$-\alpha_i\sqrt{\ell}=S_{g,X_i}=\theta_g\inv \theta_{X_i}^{-1}\theta_{g\otimes
X_i}d_{g\ot X_i}=\alpha_i\sqrt{\ell}\theta_{X_i}^{-1}\theta_{g\ot X_i},$$ which
implies that $\theta_{X_i}=-\theta_{g\otimes X_i}$. Consequently, $g\otimes
X_i\not\cong X_i$ and $g\otimes X_i\not\cong X_i^*$ so that $N_{X_i,X_i}^g=N_{X_i,X_i^*}^g=0$. This implies $X_i\ot X_i^*\cong \one\oplus\sum_j m_jY_j$, where $Y_j$ are simple
objects of dimension $2$. In particular, $d_i^2=\alpha_i^2\ell$ is odd.  Thus
$\ell$ is odd and  we conclude $\alpha_i=1$, for all $i \in \Irr\(\mcD_{1}\)$. Notice that,  since $d_i^2 = \ell$ is odd and $N_{X_i,X_i}^g=0$, we must have $X_i\ot X_i\cong \one\oplus\sum_j n_jY_j$ with $\FPdim(Y_j)=2$ so that $X_i$ is self-dual.

We claim that $\abs{\Irr\(\mcD_{1}\)}=2$. Since $g$ does not fix any $X_i$,
$\abs{\Irr\(\mcD_{1}\)}> 1$.  Let $j \in \Irr\(\mcD_{1}\)$. Since $\dim(X_i\otimes X_j)=\ell$ is odd, exactly one of $\one$ or $g$  is a subobject of $X_i\ot X_j$. Since all $X_i$ are self-dual, $N_{X_i,X_j}^\one=1$ implies $X_i\cong X_j$ whereas $N_{X_i,X_j}^g=1$ implies $X_j\cong g\otimes X_i$. Therefore, $X_i$ and $g\otimes X_i\neq X_i$
represent the only two non-isomorphic simple objects in $\mcD_1$. Thus, we
conclude that $\ell=k$, and $\abs{\Irr\(\mcD\)}=(k+7)/2$. Let $\{Y_2,\dots, Y_{(k+1)/2}\}$ be a complete set of representatives for the simple objects of dimension 2, and $\{U,V\}\subset\mcD_1$ for dimension $\sqrt{k}$. Since $\theta_g = 1$, the subcategory generated by $g$ is Tannakian.  This induces an action of $G(\mcD)$ on $\mcD$ by interchanging $\one$ and $g$, fixing all $Y_i$ and interchanging $U$ and $V$. Then the $\mbZ_2$-de-equivariantization $\mcD_{\mbZ_2}$ has $k$ objects of dimension $1$ (one from the $G(\mcC)$-orbit $\{\one, g\}$ and two from each $Y_i$) and one object of dimension $\sqrt{k}$ (from the orbit $\{U, V\}$) (cf. \cite[5.1]{M1}). As a fusion category $\mcD_{\mbZ_2}$ must be $TY(\mbZ_k,\chi,\nu)$, for which there are $4$ possible choices of $\chi$ and $\nu$. Since equivariantization and de-equivariantization are inverse operations, we have $\mcD\cong TY(\mbZ_k,\chi,\nu)^{\mbZ_2}$. As these categories are described in (b), the
proof is complete.
\end{proof}

\begin{remark}
 An alternative construction of the categories obtained in Theorem \ref{t:4m} (b) is as follows. The non-pointed factor $\mcD:=TY(\mbZ_k,\chi,\nu)^{\mbZ_2}$ can be recovered from the formula $$\mcD\boxtimes \mcP\cong\mcZ(TY(\mbZ_k,\chi,\nu)),$$ obtained from \cite[Cor. 3.30]{DMNO}. Here $\mcP$ is the maximal pointed subcategory of the Drinfeld center  $\mcZ (TY(\mbZ_k,\chi,\nu))$, which is modular by \cite[Prop. 4.2]{GNN}. There are at most $8$ possible inequivalent categories of this form up to balanced braided tensor equivalences: $4$ for the choices of $\chi$ and $\nu$ (two each) and two choices of a spherical structure.
\end{remark}

The following result will be useful later.
\begin{lemma}\label{1-non-integral-object-ising} If $\mcC$ is a weakly integral modular category
of rank $r\geq 3$ in which there is a unique simple isomorphism class of objects $X$ such that
$\FPdim X \notin \mbbZ$ then $\mcC$ is equivalent to an Ising
modular category.
\end{lemma}
\begin{proof} 
Under the hypotheses of the statement, the universal grading group $U(\mcC)\cong\mbbZ_2$ since in any faithful grading the component containing $X$ has dimension $\FPdim(X)$.  Therefore there is a (unique up to isomorphism) invertible object $g$.
Both $X$ and the non-trivial invertible object $g$ are self-dual.
The first column of the $S$-matrix is given by the Frobenius-Perron dimensions
of the simple objects of $\mcC$.  Let the index $1\in\Irr\(\mcC\)$ correspond to the non-trivial invertible object $g$.
By \cite[Lemma 4.9]{BNRW}, the Galois automorphism $\s$ that sends $\sqrt{m}$ to $-\sqrt{m}$
interchanges the first two columns of the $S$-matrix.
We see that the $S$-matrix of
$\mcC$ is:
$$\left[\begin{array}{ccccccc} 1 & 1 & d_1 & \cdots & d_t &
\sqrt{m}\\
1 & 1 & d_1 & \cdots & d_t & -\sqrt{m}\\
d_1 & d_1 &   & & & 0 \\
\vdots & \vdots & &\ddots & & \vdots \\
d_t  & d_t &   &  & & 0 \\
\sqrt{m} & -\sqrt{m} & 0 & \cdots & 0 & 0
\end{array}\right]$$
where $1<d_i\in\mbbZ$.
Applying the twist equation \eqref{twist equation} to the entry
$S_{X,X} = 0$, we get that $0= \theta_X^2
\sum_i \theta_i S_{i, X}^2=(m \theta_\1 + m \theta_g)\theta_X^2$. We conclude that $\theta_g= - \theta_\1 =
-1$. Next, we apply the twist equation \eqref{twist equation} to the entry $S_{g,X}=-\sqrt{m}$ obtaining:
$$p_+ (-\sqrt{m}) =\theta_g\theta_X\sum_i\theta_i S_{i,g}S_{i,X}=-\theta_X 2\sqrt{m}.$$
It follows that $|p_+| = 2 |\theta_X| = 2$ and $\dim \mcC = 4$. Therefore $\mcC$
is equivalent to an Ising modular category.
\end{proof}

\begin{lemma} Let $\mcC$ be a weakly integral modular category of square-free Frobenius-Perron
dimension.  Then $\mcC$ is pointed.
\end{lemma}
\begin{proof}
The only possible integral $\operatorname{FP}$-dimension of a simple object is $1$, since $\FPdim(\mcC)$ is square-free.
Therefore, the integral subcategory $\mcC_{ad}$ is pointed, hence $\mcC$ is nilpotent.
Thus, by \cite[Theorem 1.1]{DGNO1},
$\mcC$ is a Deligne product of braided subcategories of prime dimension. Such categories are pointed, by \cite[Corollary 8.30]{ENO1}.
\end{proof}

\section{General Technical Results}
The results of the first two subsections apply to arbitrary modular categories.
In particular, we do not assume that $d_i=\FPdim(V_i)$, since this cannot always
be achieved (even by changing the spherical structure).  We introduce an action
of $G(\mcC)$ on $\Irr\(\mcC\)$ and study its restriction to symmetric pointed
subcategories.  Then we explore the cycle structure of the permutation action of
$\Gal(\mcC)$ on $\Irr\(\mcC\)$, particularly for elements coming from automorphisms of $\mbbQ_{p^a}$ for primes $p$ dividing $N=\FSexp(\mcC)$.

The third subsection studies support cycles for weakly integral modular categories specifically, to be applied in our classification.

\subsection{$G(\mcC)$-grading of a modular category}\label{s:grading}
Let $G=G(\mcC)$ be the group of isomorphism classes of invertible objects of a
modular category $\mcC$.  The abelian group $G$ acts on $\Irr\(\mcC\)$ as follows:
Let $i\in G$ and $j\in\Irr\(\mcC\)$ and $V_i$, $V_j$ be representative simple objects in these classes.  Since $V_i$ is an invertible object, $V_i^*\ot V_j\cong V_k$ is simple and we define $i\cdot j=k$.

Let $\hat{G}$ denote the character group of $G$. For each $j\in\Irr\(\mcC\)$,
the map $\phi_j:i\mapsto S_{ij}/d_j$ defines a character of the Grothendieck
ring $K_0(\mcC)$. Notice that modularity implies that the restrictions of the
$\phi_j$ for all $j\in\Irr\(\mcC\)$ span $\hat{G}$. Since for $i\in G$,
$d_i=1/d_i=\pm 1$, the set of characters $\phi_0\phi_j:i\mapsto
\frac{S_{ij}}{d_id_j}$ for all $j\in\Irr\(\mcC\)$ also spans $\hat{G}$.
Thus, the modular category $\mcC$ admits a faithful $\hat{G}$-grading (cf. \cite[Theorem 6.2]{GN}) which is given by $\mcC = \bigoplus_{\chi\in\hat{G}}\mcC_\chi$, where the set of simple objects in $\mcC_\chi$ are:
$$
\Irr\(\mcC_{\chi}\) = \left\{j\in \Irr\(\mcC\) \mid \frac{S_{ij}}{d_i d_j} = \chi(i), \text{ for all } i \in G \right\}.
$$

This natural $\hat{G}$-grading on $\mcC$ induces a canonical $\hat{H}$-grading on $\mcC$, for any subgroup $H$ of $G$. More precisely, for $\chi \in \hat{H}$,
the set of simple objects in each component of the grading is:
$$
\Irr\(\mcC_{\chi}\)  = \left\{j\in \Irr\(\mcC\) \mid \frac{S_{ij}}{d_i d_j} = \chi(i), \text{ for all } i \in H \right\}.
$$
Since the restriction map $\hat{G} \to \hat{H}$ is surjective, this $\hat{H}$-grading is also faithful and $\dim \mcC_\chi = \frac{\dim \mcC}{|H|}$, for all $\chi\in\hat{H}$. See \cite[Proposition 8.20]{ENO1}.

The subcategory generated by a subgroup $H\subset G$ will be denoted $\mcC(H)$. For example $\mcC(G)=\mcC_{pt}$.  We are particular interested in the subgroups $H$ of $G$ which generates a symmetric full subcategory of $\mcC$. In this case, we will simply call $H$ a \textbf{self-centralizing} subgroup of $G$.

\begin{remark}\label{H in C_e}
If $H$ is a self-centralizing subgroup of $G$ then $S_{g,h} = d_g d_h$, for all $g,h \in H$. Therefore, $H \subset \mcC_e$, where $\mcC_e$ is the trivial component of the $\hat{H}$-grading of $\mcC$ associated to the trivial character.  Moreover if $g,h\in H$ we have
$$(d_h\theta_h)(d_g\theta_g)=(d_{h^*}\theta_{h^*})(d_g\theta_g)=d_{hg}\theta_{hg}$$ by (\ref{balancing equation}) since $S_{g^*,h}=d_{g^*}d_h$.  Therefore $\chi_H(h)=d_h\theta_h$ defines a character of $H$.
If $\mcC(H)$ is a Tannakian subcategory of $\mcC$ then $H$ is a self-centralizing subgroup and $d_h  = \theta_h = \pm 1$, for all $h \in H$, so $\chi_H$ is trivial. We will simply call such subgroups \textbf{Tannakian}. Notice that if $H$ is a self-centralizing group of odd prime order $p$ then $H$ is Tannakian by \cite[Corollary 2.50]{DGNO1}.

On the other hand, if the self-centralizing subgroup $H$ is not Tannakian, then
there exists $h \in H$ such that $d_h \theta_h = -1$ so that $\chi_H$ is a
non-trivial character. This condition provides a faithful $\mbZ_2$-grading on
the fusion subcategory $\mcC(H)$ generated by $H$. The set of isomorphism
classes $\Irr\(\mcC(H)_{0}\)$ of simple objects of the trivial component $\mcC(H)_0$ is an index 2 Tannakian subgroup. In particular, $H$ must be of even order.
\end{remark}

\begin{lemma} \label{l:grading1}
Let $H$ be a self-centralizing subgroup of $G$ and $\chi \in \hat{H}$. Then $H
\cdot j \subseteq \Irr\(\mcC_{\chi}\)$, for all $j\in\Irr\(\mcC_{\chi}\)$.
\begin{enumerate}
  \item  If $H$ is Tannakian  and $\chi$ is not trivial then $H \cdot j$ is not
    a singleton for any $j\in\Irr\(\mcC_{\chi}\)$. In particular, when $H$ is
    Tannakian of prime order and $\chi$ is not trivial, $|H|$ divides
    $\abs{\Irr\(\mcC_{\chi}\)}$.
  \item  Suppose $H$ is not Tannakian and $\chi \in \hat H$. If some $j \in
    \Irr\(\mcC_{\chi}\)$  satisfies $H \cdot j=\{j\}$ then $\chi=\chi_H$.
\end{enumerate}
\end{lemma}
\begin{proof} Since $H \subseteq \mcC_e$ and the action of $H$ is induced by the tensor product, it follows immediately from the definition of grading that $H\cdot \Irr(\mcC_\chi) \subseteq \Irr(\mcC_\chi)$.

Now, we assume that $H$ generates a Tannakian subcategory of $\mcC$. Thus $d_h
\theta_h = 1$, {for all} $h\in H$. Suppose that $j \in \Irr\(\mcC_{\chi}\)$ is fixed by $H$. 
Then $N_{h,j}^k = \delta_{k,j}$, for all $h\in H$,  $k\in \Irr\(\mcC\)$. Using the balancing equation \eqref{balancing equation}, we get that
$$\chi(h)=\frac{S_{h,j}}{d_h d_j}=  \frac{\sum_k N_{h^*,j}^k  d_k \theta_k}{\theta_h \theta_j d_h d_j} = 1 \,,  $$
{for all} $h \in H$. Therefore, $\chi$ is trivial. In addition, $|H| = p$
implies that each $H$-orbit in $\Irr\(\mcC_{\chi}\)$ has exactly $p$ classes. Therefore, the {second} statement of (i) follows.

Suppose $H$ is not Tannakian, and $j \in \Irr\(\mcC_{\chi}\)$ for some $\chi  \in \hat H$ satisfies $H \cdot j = \{j\}$. Then, for $h \in H$, the balancing equation \eqref{balancing equation} implies
$$
\chi(h)   =\frac{ S_{hj} \theta_h \theta_j }{d_h d_j \theta_h \theta_j} = \frac{d_{h\cdot j}\theta_{h\cdot j}}{d_hd_j\theta_h\theta_j}=\frac{ \theta_j d_j }{d_h d_j \theta_h \theta_j}= \frac{1}{ d_h \theta_h}  = \chi_H(h)\, ,
$$ since $\theta_h=\pm 1$ for $h\in H$.
\end{proof}
\begin{lemma}\label{l:grading3}
Suppose $H$ is a self-centralizing subgroup of $G$. If, for some $\psi \in
\hat{H}$, $\Irr\(\mcC_{\psi}\)$ consists of only one $H$-orbit then $\mcC_e$ is
integral. If, in addition, $k \in \Irr\(\mcC_{e}\)$ is fixed by $H$ then $[H:H_0]
\mid d_k$, where $H_0$ is the stabilizer of any element of $\Irr\(\mcC_{\psi}\)$.
\end{lemma}
\begin{proof}
Since $H$ is self-centralizing then $H \subseteq \mcC_e$, see Remark \ref{H in C_e}.
Let $R = \sum_{i\in \Irr\(\mcC\)} d_i V_i\in K_0(\mcC)\ot \mathbb{R}$ be the virtual regular object of
$\mcC$. Then $R= \sum_{\chi \in \hat{H}} R_\chi$, where $R_\chi = \sum_{i\in
\Irr\(\mcC_{\chi}\)} d_i V_i$ is the regular object of the component $\mcC_\chi$.

Suppose $\Irr\(\mcC_{\psi}\)=H\cdot j$, with $j\in \Irr\(\mcC_{\psi}\)$. Let
$H_0$ be the stabilizer of the $H$-orbit $\Irr\(\mcC_{\psi}\)$. Note that the quotient group $\ol H= H/H_0$ acts on $j$ as $\ol h\cdot j = h\cdot j$, where $\ol h = h H_0$. {Since $d_j = d_{h\cdot j} = d_h d_j$ for $h \in H_0$, we have $d_h=1$ for $h \in H_0$. Therefore, we can define $d_{\ol h}= d_h$ for all $h \in H$.} Thus
$R_\psi =d_j\sum_{\ol h \in \ol H} d_{\ol h} V_{\ol h \cdot j} =
\frac{d_j}{|H_0|}\sum_{h \in H} d_h V_h^* \o V_j$. Since each $\mcC_\chi$ is a
$\mcC_e$-module category, and $V_k\ot R= d_kR$, for $k \in \Irr\(\mcC_{e}\)$, we have
$V_k \ot R_\chi =d_k R_\chi$, for all $\chi \in \hat{H}$. On the other hand,
\begin{multline}
V_k \ot R_\psi = d_j \sum_{\ol h \in \ol H} d_{\ol h}V_h^* \ot V_j \ot V_k =
d_j \sum_{ \ol h, \ol{h'} \in\ol H} d_{\ol h} n_{\ol{h'}} V_h^* \ot  V_{\ol{h'}\cdot j} \\
=d_j \sum_{\ol h,\ol{h'} \in \ol H}  d_{\ol h} n_{\ol{h'}} V_{\ol{hh'}\cdot j} = d_j \sum_{h,h' \in \ol H} d_{\ol{h'}} n_{\ol{h} \ol{h'^{-1}}} V_{\ol h\cdot j}=
n R_\psi,
\end{multline}
where {$n= \sum_{\ol h \in \ol H} d_{\ol h} n_{\ol h}$ is an integer. Therefore, $d_k =n  \in \mbZ$} and $\mcC_e$ is integral, as claimed.

{If $k$ is fixed by $H$ then $d_h = 1$ for all $h \in H$ and
$$d_k R_\psi = V_k \ot R_\psi = \frac{d_j}{|H_0|}\sum_{h \in H}  V_k \o d_h V_h^* \ot V_j = d_j \frac{|H|}{|H_0|} V_j\ot V_k =
d_j \frac{|H|}{|H_0|} \sum_{\ol{h'} \in \ol H}  n_{\ol{h'}} V_{\ol{h'}\cdot j}$$
Thus, $ \frac{|H|}{|H_0|}  n_{\ol{h}}=d_k$, for all $\ol h \in \ol H$. In particular, $\frac{|H|}{|H_0|} \mid d_k$ and $n_{\ol{h}} = n_{\ol{h'}}$, for $\ol h$, $\ol{h'}\in \ol H$.}
\end{proof}
Since $\chi\in\hat{H}$ is of the form $\frac{S_{h,j}}{d_hd_j}$ for some
$j\in\Irr\(\mcC\)$, for each $\s\in\Gal(\mcC)$ we may define $(\s\circ \chi)(h)=\frac{\s(S_{h,j})}{\s(d_h)\s(d_j)}$. In the following we explore the relationship between the actions of $\Gal(\mcC)$ and $G(\mcC)$.
\begin{lem} Let $\mcC$ be a modular category and $H$ a subgroup of $G(\mcC)$.
\begin{enumerate}
  \item Suppose that $j \in \Irr\(\mcC_{\chi}\)$ for some $\chi \in \hat H$.
  Then, for $\s \in \Gal(\mcC)$ we have $\hs(j) \in \mcC_{\s \circ \chi}$. In
  particular, if $\ord(\chi)=1$ or $2$, $\hs(\Irr\(\mcC_{\chi}\)) =
  \Irr\(\mcC_{\chi}\)$.
  \item If $H$ is self-centralizing elementary 2-group,  then the actions of $H$
    and $\Gal(\mcC)$ on $\Irr\(\mcC\)$ commute.
\end{enumerate}
\end{lem}
\begin{proof}
  (i) For $j \in \Irr\(\mcC_{\chi}\)$, $\s\left(\frac{S_{ij}}{d_j}\right) =
  \frac{S_{i\hs(j)}}{d_{\hs(j)}}$ for all $i \in \Irr\(\mcC\)$. Thus, for $h \in H$,  we have
  $$
  \s(\chi(h)) = \s\left(\frac{S_{hj}}{d_h d_j}\right)=\frac{S_{h\hs(j)}}{d_h d_{\hs(j)}} \,.
  $$
  Therefore, $\hs(j) \in \Irr\(\mcC_{\s\circ \chi}\)$. If $\ord(\chi)\mid 2$, then $\chi(h) = \pm 1$ for all $h \in H$. Hence, $\s \circ \chi = \chi$ for all $\s \in \Gal(\mcC)$. The second statement of (i) follows.

  (ii) Suppose $H$ is self-centralizing elementary 2-group, and let $\chi, \psi
  \in \hat{H}$.  For $h \in H$,  $j \in \Irr\(\mcC_{\chi}\)$ and $i \in
  \Irr\(\mcC_{\psi}\)$ we have $V_i^* \o V_j \in \mcC_{\chi \psi}$ (since $H$ a
  $2$-group implies $V_i^*\in\Irr\(\mcC_{\psi}\)$) and
  $$
 \chi(h)=\frac{S_{h,j}}{d_h d_j}= \frac{\theta_{h\cdot j}} {\theta_h \theta_j}\,.
  $$
  Thus,
  \begin{eqnarray*}
    \frac{S_{i, h\cdot j}}{d_{h\cdot j}} =
   \frac{\sum_{k} N_{i^*, j}^k d_{h\cdot k} \theta_{h\cdot k}} {\theta_i \theta_{h\cdot j}d_h d_j}&=&\frac{\sum_{k} N_{i^*, j}^k d_k \chi\psi(h) \theta_h\theta_k} {\theta_i \chi(h)\theta_h \theta_j  d_j} \\
     &=&\frac{\sum_{k} N_{i^*, j}^k d_k  \theta_k\psi(h)} { \theta_i  \theta_j  d_j}
      =\frac{S_{i,j}\psi(h)} {d_j}\,.
  \end{eqnarray*}
  Therefore,
  $$
   \frac{S_{i, h\cdot \hs(j)}}{d_{h\cdot \hs(j)}} = \frac{S_{i,\hs(j)}\psi(h)}{d_{\hs(j)}}= \s\left(\frac{S_{i,j} \psi(h)} {d_j}\right) =\s\left(\frac{S_{i, h\cdot j}}{d_{ h\cdot j}}\right)
   =\frac{S_{i, \hs(h\cdot j)}}{d_{\hs( h\cdot j)}}\,.
  $$
  Since $\{\phi_j(i):=\frac{S_{i,j}}{d_j}\mid j\in\Irr\(\mcC\)\}$ is a linearly
  independent set of characters of $K_0(\mcC)$, this implies $h\cdot
  \hs(j)=\hs(h\cdot j)$ for all $j\in\Irr\(\mcC\)$, i.e. the actions of $H$ and
  $\Gal(\mcC)$ on $\Irr\(\mcC\)$ commute.
\end{proof}

\subsection{Support cycles}\label{s:support}
Let $n$ be a positive integer, and $p$ an odd prime number. Then, there are unique non-negative integers $a, q$ with $p \nmid q$ such that $n=p^a q$. For the remainder discussion of this paper, we will abbreviate $v_p(n) =a$ and
$k_p(n) =\frac{\varphi(p^a)}{2}$, where $\varphi$ is the Euler's totient function. For any element $x$ in a group of finite order, we will denote $\ord_p(x) = p^{v_p(\ord(x))}$, the maximal $p$-power factor of $\ord(x)$.

 Let $N=\FSexp(\mcC)$ of a modular category $\mcC$. By \cite{DLN1}, the level $n=\ord(t)$ of any normalized modular data $(s,t)$ satisfies $N \mid n \mid 12N$. In particular, the modular data $(s,t)$ is defined over the cyclotomic field $\mbQ_{12 N}$.

 Suppose $p$ is an odd prime and  $12 N=p^b m$  for some positive integer $b$ and $m$ with $p \nmid m$. There exists $\s \in \Gal(\mbbQ_{ab})$, where $\mbbQ_{ab}$ is the abelian closure of $\mbbQ$ in $\mbbC$, such that  $\s$ fixes $\mbbQ_{m}$ and $\s|_{\mbbQ_{p^{b}}}$ is a generator of the cyclic group  $\GalQ{p^{b}}\cong \mbbZ_{p^a}^*$.  We call such $\s$ \textbf{$p$-automorphisms} of $\mcC$.

 Associated with a $p$-automorphism $\s$ is a permutation $\hs$ on $\Irr\(\mcC\)$, which can be decomposed into a product of unique disjoint cycles. For notational convenience we identify a cycle with its support.  Suppose $C$ is a cycle in this decomposition and $j \in C$ and $\ord_p(t_j) = \ell$, where $t_j$ is the $j$-th diagonal entry of $t$ in a normalized modular data $(s,t)$. Then, by Galois symmetry (cf. \cite{DLN1}),
 $$
 \ell = \ord_p(\s^2(t_j)) = \ord_p(t_{\hs(j)}).
 $$
  Thus, $\ord_p(t_i) = \ell$ for all $i \in C$, and we call $C$ a  \textbf{$p^\ell$-support cycle} of $\s$ if $\ell > 0$. If $a = v_p(n)> 0$, a $p^a$-support cycle of $\s$ is called a \textbf{maximal} support cycle of $\s$.
\begin{remark}
If $(s',t')$ is another normalized modular data of $\mcC$, then $(s',t')=(sx^{-3}, tx)$ for some $12$-th root of unity $x$.
Therefore,  all the $p^\ell$-support cycles are \emph{independent} of the choice of the normalized pair $(s,t)$ if $p> 3$ or $\ell > 1$. However, the $1$ and $3$-support cycles of a $3$-automorphism could be changed if $3\mid \ord(x)$. We can resolve this particular ambiguity by using only a fixed family of normalizations for some modular categories.
\end{remark}
\begin{lemma}\label{l:tech1}
Suppose  $\mcC$ is a modular category and $(s,t)$ a normalized modular data of level $n=p^a q$ for some odd prime $p$ and positive integers $a, q$ such that $p\nmid q$. Then $\s$ admits a maximal support cycle. For any  $p^\ell$-support cycle $C$  of  $\s$ with $1 \le \ell \le a$,  we have
$$
\frac{\varphi(p^\ell)}{2} \mid \ord(C) \mid 2k_p(n)=\varphi(p^a).
$$
In particular, for the maximal support cycle we have
$k_p(n) \mid \ord(\hs)  \mid 2k_p(n)$.

If $1$ is in a $p^\ell$-support cycle $C$ of $\s$ and $t_1= x\z$ for some $x,\z \in \mbC$ with $\ord(\z)=p^\ell$ and $x^q=1$,   then $t_{\hs^j(1)} = x\s^{2j}(\z)$ for all $j$. In particular, $\ord_p(t_j)=p^\ell$ for all $j \in C$.
\end{lemma}
\begin{proof}  Let $k=k_p(n)$. Since $\ord(\s|_{\mbbQ_{p^a}}) = 2k$, $\ord(\hs) \mid 2 k$ and  the length of each disjoint cycle of $\hs$ is a divisor of $2k$.

  Since $\ord_p(n)=a$, there exists $j_0 \in \Irr\(\mcC\)$ such that $\ord_p(t_{j_0}) = p^a$. Therefore, the disjoint cycle of $\hs$ containing $j_0$ is a $p^a$-support cycle of  $\s$.

 Suppose $C$ is a $p^\ell$-support cycle of $\s$, for some $\ell \le a$, with $1 \in C$. Then $t_1 = x\z$ for some $x, \z \in \mbC$ such that $\ord(\z)=p^\ell$ and $x^q=1$.  Let $k' = \varphi(p^\ell)/2$. Note that
$$
   t_1 , \s^{2}(t_1), \dots,\s^{2(k'-1)}(t_1)
$$
 are distinct, and $\s^{2k'}(t_1) = t_1$.   By Galois symmetry,
 $$
  \s^{2j}(t_1) =t_{\hs^j(1)}  =  x  \s^{2j}(\z)  \text{ for all } j.
 $$
 In particular,  $1, \hs(1), \dots, \hs^{k'-1}(1)$ are distinct, and $\ord_p(t_{\hs^j(1)})=p^\ell$. Therefore,  the length of the cycle $C$  is a multiple $k'$ and so
$$
k' \mid \ord(C) \mid \ord(\hs) \mid 2k\,.
$$
For $\ell=a$, we find $k \mid \ord(C) \mid  \ord(\hs) \mid 2k$.
\end{proof}
If the anomaly $\a=p_+/p_-$ of $\mcC$ is such that $3 \nmid \ord(\a)$, one can choose a $6$-th root $\l$ of $\a$ such that $\frac{p_+}{\l^3} =\sqrt{\dim \mcC}$ and $3 \nmid \ord(\l)$. The $3$-support cycle of a $3$-automorphism of $\mcC$ is independent of the choice of any of these normalizations. Moreover, the level $n$ of any of these normalized modular satisfies $\FSexp(\mcC) \mid n \mid 4 \FSexp(\mcC)$.

\subsection{Support cycles for weakly integral modular categories}\label{s:support weakly integral}

We are particularly interested in weakly integral modular categories as their anomaly $\a = \frac{p_+}{p_-}$ can only be an $8$-th root of unity (cf. \cite{DLN1}). For any weakly integral modular category $\mcC$, we only consider the normalized modular data $(s,t) = (\frac{S}{D}, \frac{1}{\l}T)$ with $\l= \sqrt{\a}^5$ where $\sqrt{\a}D = p_+$.  A $3$-support cycle for a $3$-automorphism of $\mcC$ for this normalized modular data is then well-defined and we have a more refined statement for $p^\ell$-support cycles for a weakly integral modular category.
\begin{lem} \label{l:tech2}
Suppose $p$ is an odd prime and $\s$ is a $p$-automorphism of a weakly integral modular category  $\mcC$ such that $v_p(\FSexp(\mcC))=a> 0$.  Then,  the following statements are equivalent for any positive integer $\ell \le a$:
\begin{enumerate}
\item $C$ is a $p^\ell$-support cycle of $\s$;
\item $\ord_p(\theta_{j})=p^\ell$ for some $j \in C$;
\item $\ord_p(\theta_{j})=p^\ell$ for all $j \in C$.
\end{enumerate}
If $C=(1\, 2\,\dots\, l )$ and $\theta_1 = x \z$ for some $x, \z \in \mbC$ with $\ord(\z)=p^\ell$ and $x^q=1$ with $p\nmid q$, then $\theta_j = x\s^{2j}(\z)$ for $j=1, \dots, l$. In particular, if $C_0$ is the cycle of $\s$ containing $0$, then $\theta_j =1$ for all $j \in C_0$ and $\ord(C_0) \le 2$. Moreover, $V_{\hs(0)}$ is a self-dual invertible object.
\end{lem}
\begin{proof}
  Let $t = \sqrt{\a}^5 T$ where $\sqrt{\a}= p_+/D$ is a $16$-th root of unity. Then $\ord_p(t_j) = \ord_p(\theta_j)$ for all $j$, and so the equivalence of the three conditions follows from Lemma \ref{l:tech1}. Note that $t_0$ is a $16$-th root but $p \mid \ord(t_j)$ for any $j$ in a $p^\ell$-support cycle $C$ of $\s$ ($\ell >0$). Therefore, $0 \not\in C$. If $C=(1\, 2\,\dots\, l )$ and  $\theta_1 = x \z$ for some $x, \z \in \mbC$ with $\ord(\z)=p^\ell$ and $x^q=1$, then $t_1 = \sqrt{\a}^5 x \z$ and so $t_j= \sqrt{\a}^5 x\s^{2j}(\z)$. Thus, $\theta_j = x\s^{2j}(\z)$.

  Since $\mcC$ is weakly integral, $\s^2(d_i) = d_i$ for all $i \in
  \Irr\(\mcC\)$. This implies the length of $C_0$ is at most 2. Since
  $\theta_0=1$, $\theta_{\hs(0)} = 1$. If $g=\hs(0) \ne 0$, then $S_{g,j} = \pm
  S_{0,j} = \pm d_j$ for all $j \in \Irr\(\mcC\)$. By \cite[Lem. 6.1]{GN}, $V_g$ is invertible. Since $g$-th row of the $S$-matrix is real, $V_g$ is self-dual.
 \end{proof}
 \begin{theorem}\label{t:gal}
Let $\mcC$ be a weakly integral modular category with $N=\FSexp(\mcC)$ which has a simple odd prime factor $p$.
 Suppose a $p$-automorphism $\s$ of $\mcC$ has only one $p$-support cycle $C_1=(1\,\dots\, l)$. Then:
\begin{enumerate}
  \item $\dim \mcC= \(\frac{l}{p-1}\)^2 d_1^4 p$, $\s(d_1)=d_1$, $\s(p_+)=-p_+$ and $\theta_1 = x \z$, for some $x, \z \in \mbC$ with $\ord(\z)=p$ and $x^{16}=1$. In particular, if $l=\frac{p-1}{2}$, then  $2 \mid d_1^2$.
  \item If $\mcC$ is integral and $\Gal(\mcC)$ is generated by $\hs$, then $\hs$ is a cycle of length $p-1$ and $\mcC$ is pointed of rank $p$. In particular, $\mcC$ is a cyclic modular category.
  \item If $\mcC$ is a strictly weakly integral and $\Gal(\mcC)$ is generated by $\hs$, then $\hs=(0,1)(2,\dots, \frac{p+1}{2})$, up to a relabeling of the simple objects, and $\mcC$ is a prime modular category of $\dim\mcC = 4p$. In particular, $\mcC$ is of rank $\frac{p+7}{2}$, and equivalent to a $\mbZ_2$-equivariantization of a Tambara-Yamagami category associated to the group $\mbZ_p$.
\end{enumerate}
\end{theorem}
\begin{proof} Since $\mcC$ is pseudo-unitary, we may assume $d_i > 0$ for all $i
  \in \Irr\(\mcC\)$ by using the canonical pivotal structure of $\mcC$. Suppose $N=pq$ for some non-negative integer $q$ relatively prime to $p$. Then  $\theta_1 = x \z$ for some $x, \z \in \mbC$ with $\ord(\z)=p$ and $x^q=1$.  By Lemma \ref{l:tech2},  $\theta_j = x\s^{2j}(\z)$, for $j=1, \dots, l$, and $l = p-1$ or $\frac{p-1}{2}$. Suppose $\hs = C_0 C_1 C_2 \cdots C_m$ where $C_0$ is the cycle containing $0$. Note that $d_j=d_{C_i}$  for all $j \in C_i$. In particular,  $d_{C_1}=d_1$. Since $C_1$ is the only $p$-support cycle of $\s$, we also have $\theta_j=\theta_{C_i} \in \mbbZ[\zeta_q]$  for all $j \in C_i$ if $i \ne 1$.  Let us denote $l_i = \ord(C_i)$ and $C_i=(c_{i}\, \hs(c_{i})\,\hs^2(c_i)\,\dots)$. Now we consider the twist equation \eqref{twist equation} for $S_{0j}$:
  \begin{eqnarray}
 p_+ d_j \ol\theta_j  & = &  d_1  \sum_{r=1}^l  S_{jr}  \theta_r+  \sum_{i\ne 1}  d_{C_i}\theta_{C_i} \sum_{r\in C_i} S_{jr} \label{eq:quadratic0} \\
 & = &  d_1^2 x \sum_{r=0}^{l-1} \frac{S_{j, 1+r}}{d_1} \s^{2r}(\z)+  \sum_{i\ne 1}  d_{C_i}^2\theta_{C_i} \sum_{r\in C_i} \frac{S_{jr}}{d_r}\,. \label{eq:0.2}
 \end{eqnarray}
For any $i$, we denote
 $$
S_{j, C_i}= \sum_{r \in C_i} \frac{S_{jr}}{d_r} = \sum_{r=0}^{l_i-1} \s^{r-1}\(\frac{S_{j, c_i}}{d_{C_i}} \) .
$$
Then $\s$ fixes $S_{j, C_i}$, and so
 $\sum_{i \ne 1} d_{C_i}^2 \theta_{C_i} S_{j, C_i}$, and \eqref{eq:0.2} becomes
 \begin{equation}\label{eq:0.3}
    p_+ d_j \ol\theta_j =  d_1^2 x \sum_{r=0}^{l-1} \s^r\left(\frac{S_{j, 1}}{d_1}\right) \s^{2r}(\z)+  \sum_{i\ne 1}  d_{C_i}^2\theta_{C_i} S_{j, C_i}\,.
 \end{equation}

 Therefore, for $j=0$,   we have
 \begin{eqnarray}\label{eq:quadratic1}
 \sqrt{\a}D = p_+
 & = &  d_1^2 x \sum_{r=0}^{l-1} \s^{2r}(\z)+  \sum_{i \ne 1} d_{C_i}^2 \theta_{C_i} l_i \label{eq:0.4}\\
             & = & \frac{2 l d_1^2 x}{p-1} \(\frac{-1\pm \varepsilon \sqrt{p}}{2}\)  +  \sum_{i \ne 1} d_{C_i}^2 \theta_{C_i} l_i  \\
             & = &  \pm  \frac{l d_1^2 x}{p-1}\varepsilon \sqrt{p}  +  \frac{-l d_1^2 x}{p-1} + \sum_{i \ne 1} d_{C_i}^2 \theta_{C_i} l_i  \label{eq:0.5}
 \end{eqnarray}
 where $\varepsilon=\sqrt{\(\frac{-1}{p}\)}$ and $\(\frac{\cdot}{p}\)$ is the Legendre symbol. Since $\{1, \sqrt{p}\}$ is linearly independent over $\mbbQ_q$, it follows from \eqref{eq:0.5} that
\begin{equation}\label{eq:00}
\sqrt{\a} D = \pm \frac{d_1^2 x l}{p-1} \varepsilon \sqrt{p} \quad \text{and}\quad
\frac{l d_1^2 x}{p-1} =  \sum_{i \ne 1} d_{C_i}^2 \theta_{C_i} l_i
\,.
\end{equation}
This implies $\a = x^2 \(\frac{-1}{p}\)$,  $\dim \mcC = \(\frac{l}{p-1}\)^2 d_1^4 p$, and $\s(p_+)=-p_+$. Therefore,
\begin{equation} \label{eq:p1}
2 p_+ = d_1^2 x \sum_{r=0}^{l-1} \(\s^{2r}(\z) - \s^{2r+1}(\z)\)\,.
\end{equation}
We now show that $d_1$ is fixed by $\s$.
Suppose $\s(d_1) \ne d_1$, then $\s(d_1)=-d_1$ as $d_1$ is square root of an
integer. Since $\s^2(d_j)=d_j$ for all $j \in \Irr\(\mcC\)$, we get $l_0=2$ and we write $C_0 = (0, \hs(0))$. Thus,
$$
S_{1, C_0} = d_1+S_{1,\hs(0)} = d_1 + (-d_1) =0\,.
$$
By Lemma \ref{l:tech2}, $\theta_{\hs(0)}=1$. Summing \eqref{eq:0.3} for $j=0, \hs(0)$, we find
$$
  2 p_+  =  \sum_{i\ne 1} d^2_{C_i}\theta_{C_i} (S_{0,C_i}+S_{\hs(0), C_i}) \,.
$$
However, this implies $p_+$ is fixed by $\s$, which is a contradiction. Therefore, $\s$ fixes $d_1$.
This completes the proof of (i).

We now assume $\Gal(\mcC)=\langle \hs \rangle$ for the remainder proofs of (ii) and (iii).
Then, it follows from (i) that $d_1, S_{j,C_i} \in \mbbZ$
for all $i$, and $d_{C_i}^2 \theta_{C_i} S_{j,C_i} \in \mbbZ[\zeta_q]$ for $i \ne 1$.

Let $I = \{j \in \Irr\(\mcC\) \mid d_i \in \mbZ\}$, $I^c = \Irr\(\mcC\) \setminus I$ and $\tilde I= I \setminus \{1, \dots, l\}$, and let
$\langle I \rangle$ be the tensor subcategory generated by $I$. For any disjoint cycle $C_a$ of $\s$,  we write $C_a \subset \tilde I$ if all the entries of $C_a$ are in $\tilde I$. Similarly, we write $C_a \subset I^c$ if all the entries of $C_a$ are in $I^c$.

Suppose $C_a \subset \tilde I$. By summing $j \in C_a$ of \eqref{eq:0.3}, we find
\begin{equation*}
  l_a p_+ d_{C_a} \ol \theta_{C_a}  = d_1 d_{C_a} x S_{1, C_a}  \sum_{r=0}^{l-1} \s^{2r}(\z) + \sum_{i\ne 1 \atop j \in C_a }  d_{C_i}^2\theta_{C_i} S_{j, C_i}\,.
\end{equation*}
Since the last term of this equation is fixed by $\s$, the difference of this equation and its image under $\s$ is given by
\begin{equation*}
 2 l_a p_+ d_{C_a} \ol \theta_{C_a}  =  d_1 d_{C_a} x S_{1, C_a}  \sum_{r=0}^{l-1} \s^{2r}(\z) - \s^{2r+1}(\z)
\end{equation*}
By \eqref{eq:p1}, we find
$
S_{1, C_a}  = l_a d_1  \ol\theta_{C_a}\,.
$
Hence $\theta_{C_a} = \pm 1$ and $S_{r, C_a}= l_a d_1  \ol\theta_{C_a}$, for all $r\in C_1$.
Since $|S_{r,j}|\le d_1 d_{C_a}$ for $j \in C_a$, we find
\begin{equation}\label{eq:central1}
  S_{r, j}=d_r d_j\theta_{C_a}  \in \mbbZ,
\end{equation}
with $j \in C_a$ and $r \in C_1$.

Suppose $C_b \subset \tilde I$. For $j \in C_b$, we have $S_{j, C_a} = S_{c_b, C_a}$.
For any $j' \in C_a$, we have the twist equation:
$$
p_+ S_{j,j'} =   \theta_j \theta_{j'}\(  x  \sum_{r=0}^{l-1} S_{j,1+r}S_{j',1+r} \s^{2r}(\z)  +  \sum_{i \ne 1} \sum_{r \in C_i} S_{jr}  S_{j'r}\theta_{C_i}\)\,.
$$
Therefore,
\begin{eqnarray*}
  &&\sum_{j \in C_b}\sum_{j' \in C_a} p_+ S_{j,j'}  =   d_{C_a}   l_b p_+ S_{c_b ,C_{a}} \\
  &=&
   \theta_{C_b} \theta_{C_a}\( x d_{C_a} d_{C_b}\sum_{r=0}^{l-1} S_{1+r, C_b}S_{1+r, C_a} \s^{2r}(\z)  +  d_{C_a}d_{C_b}\sum_{i\ne 1} \sum_{r \in C_i} S_{r, C_b}  S_{r, C_a}\theta_{C_i}\)\\
 &=&
   \theta_{C_b}  \theta_{C_a}\( x d_{C_a} d_{C_b} S_{1, C_b}S_{1, C_a} \sum_{r=0}^{l-1}\s^{2r}(\z)  +  d_{C_a}d_{C_b}\sum_{i\ne 1} \sum_{r \in C_i} S_{r, C_b}  S_{r, C_a}\theta_{C_i}\)\\
   &=&
    x d_{C_a} d_{C_b} l_a l_b d_1^2  \sum_{r=0}^{l-1}\s^{2r}(\z)  +  \theta_{C_b}  \theta_{C_a} d_{C_a}d_{C_b}\sum_{i\ne 1} \sum_{r \in C_i} S_{r, C_b}  S_{r, C_a}\theta_{C_i}\,.
\end{eqnarray*}
Since $S_{c_b ,C_{a}}$ and $\theta_{C_b}  \theta_{C_a}d_{C_a}d_{C_b}\sum_{i\ne 1} \sum_{r \in C_i} S_{r, C_b}  S_{r, C_a}\theta_{C_i}$ are fixed by $\s$, we find
$$
 2  d_{C_a}  l_b p_+ S_{c_b ,C_{a}} =
x d_{C_a} d_{C_b} l_a l_b d_1^2 \sum_{j=0}^{l-1} \(\s^{2j}(\z) - \s^{2j+1}(\z)\)= 2p_+d_{C_a} d_{C_b} l_a l_b\,.
$$
Hence, $S_{c_b, C_a} = l_a d_{C_b}$. By the same argument as above, $S_{j,j'} = d_{C_a} d_{C_b} = d_j d_{j'}$ for $j \in C_b$ and $j' \in C_a$. For any $j,j' \in \tilde I$, there exist $C_a, C_b \subset \tilde I$ such that $j' \in C_a$ and $j \in C_b$. Hence we have
\begin{equation} \label{eq:central2}
 S_{j,j'}  = d_j d_{j'}, \text{ for all } j, j' \in \tilde I\,.
\end{equation}

By summing \eqref{eq:0.3} over all $j \in C_b$, we obtain
\begin{eqnarray*}
  l_b p_+ d_{C_b} \theta_{C_b} & = &  d_1 d_{C_b} x \sum_{r=0}^{l-1} S_{r, C_b} \s^{2r}(\z) +  \sum_{j \in C_b} \sum_{i\ne 1} d^2_{C_i}\theta_{C_i} S_{j, C_i} \\
  & = &  d_1 d_{C_b} x S_{1, C_b} \sum_{r=0}^{l-1}  \s^{2r}(\z) +  l_b d_{C_b}\sum_{C_i \subset \tilde I} d^2_{C_i}\theta_{C_i} l_i  +  \sum_{j \in C_b} \sum_{C_i \subset I^c } d^2_{C_i}\theta_{C_i} S_{j, C_i} \\
  & = & l_b d_1^2 d_{C_b} \theta_{C_b} x  \sum_{r=0}^{l-1}  \s^{2r}(\z) +    l_b d_{C_b}\sum_{C_i \subset \tilde I} d^2_{C_i}\theta_{C_i} l_i  +  \sum_{j \in C_b} \sum_{C_i \subset I^c } d^2_{C_i}\theta_{C_i} S_{j, C_i}
\end{eqnarray*}
or
\begin{equation} \label{eq:0.7}
  p_+ = d_1^2  x  \sum_{r=0}^{l-1}  \s^{2r}(\z) +    \theta_{C_b}\sum_{C_i \subset \tilde I} d^2_{C_i}\theta_{C_i} l_i  +  \frac{\theta_{C_b}}{l_b d_{C_b}} \sum_{C_i \subset I^c \atop {j \in C_b} } d^2_{C_i}\theta_{C_i} S_{j, C_i}
\end{equation}

(ii) If $\mcC$ is integral, then $I= \Irr\(\mcC\)$ and so
$$
 p_+   =   d_1^2  x  \sum_{r=0}^{l-1}  \s^{2r}(\z)  +  \theta_{C_b}\sum_{i \ne 1} d^2_{C_i}\theta_{C_i} l_i\,.
$$
Since $\sum_{i \ne 1} d^2_{C_i}\theta_{C_i} l_i \ne 0$ (cf. \eqref{eq:00}), by comparing this equation with
\eqref{eq:0.4}, we have $\theta_{C_b}=1$ for $C_b \in \tilde I$. It follows from \eqref{eq:central1} that $S_{r,j} = d_r d_j$ for
$r \in C_1$. Therefore, if $j \not\in C_1$, $S_{i,j} = d_i d_j$ for all $i \in
\Irr\(\mcC\)$. This forces $j=0$ and so $\Irr\(\mcC\) = \{0, \dots l\}$. By (i) we have
$$
\(\frac{l}{p-1}\)^2 d_1^4 p = 1+ l d_1^2\,.
$$
The equation has no integral solution for $d_1$ if $l = \frac{p-1}{2}$. For $l=p-1$, the only integer solution is $d_1 =1$. Thus, $\mcC$ is pointed of rank $p$, and the proof of (ii) is completed.

(iii) Now, we assume  $\mcC$ is strictly weakly integral.  Then $I^c \ne \emptyset$ and so $C_0$ must have length $> 1$. By Lemma \ref{l:tech2}, $C_0 = (0, g)$, $\theta_g=1$ and $g\in G(\mcC)$ is of order $2$,  where $g = \hs(0)$. Let $H=\{0,g\}$ and $\mcC(H)$ the fusion category generated by $V_0$ and $V_g$. Then $\dim \mcC(H) =  2$ and $\mcC(H)$ is a subcategory of the centralizer ${\langle I \rangle}'$ by \eqref{eq:central1} and \eqref{eq:central2}, where $\langle I \rangle$ is the subcategory generated by $V_i,i\in I$. We proceed to show that the centralizer ${\langle I \rangle}'=\mcC(H)$ by showing that $\dim {\langle I \rangle}' =2$.

Since $\Gal(\mcC)$ is cyclic, $\mbK_\mcC$ has exactly one quadratic extension over $\mbQ$. Therefore, there exists a square-free positive $M$ such that $d_j=\a_j \sqrt{M}$, for all $j \in I^c$ and positive integer $\a_j$. Let $\mcC_0=\langle I\rangle$ and let $\mcC_1$ be the subcategory of $\mcC$ generated by $V_j$ for $j \in I^c$. Then,
$\mcC = \mcC_0 \oplus \mcC_1$ defines a faithful $\mbZ_2$-grading on $\mcC$. Thus, $\dim \langle I \rangle = \dim \mcC/2$. By \cite{M2}, $\dim {\langle I\rangle}' = 2$.

Next, we would like to prove that $\tilde I = H$.  Now we consider the twist equation \eqref{twist equation} for $S_{g,j}$ with $j \in \tilde I$:
  \begin{eqnarray}
 p_+ d_j \ol\theta_j  & = &  d_1  \sum_{r=1}^l  S_{j,r}  \theta_r+  \sum_{C_i \subset \tilde I}  d_{C_i}\theta_{C_i} \sum_{r\in C_i} S_{j,r} -  \sum_{C_i \subset I^c}  d_{C_i}\theta_{C_i} \sum_{r\in C_i} S_{j,r}\,.
\end{eqnarray}
By comparing this equation with \eqref{eq:quadratic0}, we find
\begin{equation}
  \sum_{C_i \subset I^c}  d^2_{C_i}\theta_{C_i} S_{j,C_i} = \sum_{C_i \subset I^c}  d_{C_i}\theta_{C_i} \sum_{r\in C_i} S_{j,r} = 0 \text{ for all } j \in \tilde I.
\end{equation}
By setting $j=0$ in \eqref{eq:0.3}, we obtain $\sum_{C_i \subset I^c}  d^2_{C_i}\theta_{C_i} l_i =0$. Therefore, by comparing \eqref{eq:0.4} and \eqref{eq:0.7}, we have $\theta_j =1$, for all $j \in \tilde I$. This implies $S_{ij}=d_i d_j$, for all $i \in I$ and $j \in \tilde I$. Hence $V_j \in C_\mcC(\langle I \rangle)$, for $j \in \tilde I$. Therefore, $\tilde I = H$ and so $I = H \cup \{1, \dots, l\}$.

By (i), we now have an equation on the dimension of $\langle I \rangle$:
$$
\(\frac{l}{p-1}\)^2 \frac{d_1^4 p}{2} = \frac{\dim \mcC}{2} = 2 + l d_1^2\,.
$$
This equation has no integral solution  for $d_1$ if $l =p-1$. This implies $l=\frac{p-1}{2}$, and so $d_1 =2$ is the only integer solution. Thus, $\dim \mcC = 4p$. It follows from Theorem \ref{t:4m} that $\mcC$ is a $\mbZ_2$-equivariantization of a Tambara-Yamagami category associated to the group $\mbZ_p$. This completes the proof of (iii).
\end{proof}

\begin{lemma}\label{l:tech3}
Suppose $\mcC$ is a weakly integral modular category such that $v_p(\FSexp(\mcC))=1$, for some odd prime $p$. Let $\s$ be a $p$-automorphism of $\mcC$. If $\s$ has exactly two $p$-support cycles and they both have length $k=\frac{p-1}{2}$, say $C_1= (1\,\dots\, k)$ and $C_2=(k+1\, \dots\, 2k)$, then
either
  \begin{enumerate}
\item $v_p(\dim \mcC)$ is even and $d_1 =\cdots=d_{p-1}$, or
\item $v_p(\dim \mcC)$ is odd and
$\dim \mcC = \frac{1}{4}\(d_1^4+d_{2k}^4+ \e d_1^2 d_{2k}^2\)p$, for some $\e \in \{0,1,-1\}$.
\end{enumerate}
\end{lemma}
\begin{proof}
Suppose $\hs= (1\,\dots\, k)(k+1\, \dots\, 2k) \Prod_{i} Z_i$, for some disjoint
cycles $Z_i$. Let $J = \Irr\(\mcC\)\setminus \{1,\dots, 2k\}$. Since none of the elements of $J$ are contained in any support cycles of $\s$, it follows that $\s(\theta_j) = \theta_j$, for $j \in J$.
Now we assume $\FSexp(\mcC)=pq$, $\theta_1= x \z$ and $\theta_{2k}= y \z^\ell$, for some $x,y ,\z\in \mbC$ and some positive integer $\ell < p$ such that $x^q=y^q=1$ and $\ord(\z)=p$. Define $u = \sum_{i=1}^{k-1} \s^{2i}(\z)=\frac{-1\pm \varepsilon \sqrt{p}}{2}$, $u' = \sum_{i=1}^{k-1} \s^{2i}(\z^\ell)=\frac{-1\pm \(\frac{\ell}{p}\)\varepsilon \sqrt{p}}{2}$, where $\varepsilon =\sqrt{\(\frac{-1}{p}\)}$. Then
\begin{eqnarray*}
  \sqrt{\a} D = p_+  & = & d^2_1 x u + d^2_{2k }y u' +\sum_{j\in J} d_j^2 \theta_j \\
  & = & \frac{\pm 1}{2} (d^2_1 x  +   \(\frac{\ell}{p}\) d^2_{2k}y) \varepsilon\sqrt{p} -\frac{1}{2} (d^2_1 x  +  d^2_{2k}y)+\sum_{j\in J} d_j^2 \theta_j\,.
\end{eqnarray*}
Note that $ \sum_{j\in J} d_j^2 \theta_j \in \mbbQ_{q}$. If $v_p(\dim \mcC)$ is even, then $p_+ \in \mbbQ_{q}$. The
$\mbbQ_{q}$-linear independence of $\{1,\sqrt{p}\}$ implies  $d^2_1 x  +  \(\frac{\ell}{p}\)   d^2_{2k}y=0$, and so $d^2_1   =  d^2_{2k}$. On the hand, If $v_p(\dim \mcC)$ is odd, then $p_+\not\in \mbbQ_q$, and so
$$
\sqrt{\a} D = \frac{\pm 1}{2} (d^2_1 x  +   \(\frac{\ell}{p}\) d^2_{p-1}y) \varepsilon\sqrt{p} \,.
$$
Hence, by considering the product of conjuagates, we have
$$
 \frac{4 \dim \mcC}{p} = \(d_1^2  + \(\frac{\ell}{p}\) d^2_{2k} \frac{y}{x}\) \(d_1^2  + \(\frac{\ell}{p}\) d^2_{2k} \frac{y}{x}\) =
 d_1^4+d_{2k}^4+\(\frac{\ell}{p}\) d_1^2 d_{2k}^2\(\frac{y}{x}+\frac{x}{y}\)\,.
 $$
 This implies $y/x$ is either a 4-th root or a 6-th root of unity and so we obtain
 $$
 \frac{4 \dim \mcC}{p} =
 d_1^4+d_{2k}^4+ \e d_1^2 d_{2k}^2\, ,
 $$
 where $\e \in\{0,1,-1\}$.
\end{proof}

\section{Weakly integral modular categories of rank 6 and 7}
We now apply the results of previous sections to the classification of weakly integral modular categories of ranks $6$ and $7$.

\subsection{Weakly integral modular categories of rank 6}
\begin{theorem}
A weakly integral rank $6$ modular category $\mcC$ is equivalent (as balanced braided fusion category) to one of the following:
\begin{enumerate}
\item[(a)] $\mcI\boxtimes\mcP$, with $\mcI$ an Ising modular category and $\mcP$ a cyclic modular category of rank $2$,

\item[(b)] $TY(\mbZ_5,\chi,\nu)^{\mbZ_2}$, or
\item[(c)] a cyclic modular category of rank $6$.
\end{enumerate}
\end{theorem}
\begin{proof}

If $\mcC$ is integral, it follows from \cite[Theorem 4.2]{BR1} that $\mcC$ is in fact pointed (alternative (d)). Therefore, we may assume that $\mcC$ is strictly weakly integral. Moreover, by Lemma \ref{1-non-integral-object-ising}, there are at least two objects with non-integral dimensions. In addition, \cite[Theorem 3.10]{GN} implies that $\mcC$ is faithfully graded by an elementary $2$-group, so that there are at least two invertible objects.
Thus, the universal grading group $U(\mcC)$ has order $2$ or $4$.

If $|U(\mcC)|=4$ then the two simple objects with non-integral dimension must have dimension $\sqrt{2}$.  In this case $\mcC$ is equivalent to $\mcI\boxtimes\mcP$, with $\mcP$ pointed of rank $2$ by \cite[Theorem 5.5]{Nat1}.

If $|U(\mcC)|=2$ it is enough to show that $\dim\mcC\in\{12,20\}$ by Theorem \ref{t:4m}. Note that $\mcC_{ad}=\mcC_{e}$ is an integral premodular category containing exactly two isomorphism classes of invertible objects. Since there are at least two simple classes of objects of non-integral dimension, the rank of $\mcC_{ad}$ is $2,3$ or $4$. The first case is clearly impossible, since there are two pointed objects in $\mcC_{ad}$. If $\mcC_{ad}$ has rank $3$, $\dim\mcC=2\dim\mcC_{ad}=12$ by Ostrik's classification of rank $3$ premodular categories \cite{O4}. If $\rank(\mcC_{ad}) = 4$ we have $\dim\mcC=2\dim\mcC_{ad}=20$, by the classification of rank 4 premodular categories \cite{B2}.
\end{proof}

\subsection{Weakly Integral Rank 7}

The goal of this subsection is to give a classification, up to Grothendieck equivalence, of weakly integral rank $7$ modular categories.  Such categories obviously cannot be Deligne products of non-trivial categories.
Applying the results of subsection \ref{s:grading} we have:
\begin{prop} \label{p:bosonic_obj}
Suppose $\mcC$ is a (not necessarily weakly integral) modular category of rank $7$. If $\mcC$ admits a Tannakian subcategory equivalent to $\Rep(H)$, for some nontrivial subgroup $H$ of $G(\mcC)$, then $\mcC$ is strictly weakly integral of dimension $28$.
\end{prop}
\begin{proof}
Since $\mcC$ has a faithful $\hat{H}$-grading and $H\subset\mcC_e$ by Remark
\ref{H in C_e}, the pigeonhole principle implies that $|H| \le 4$. If $|H|=4$
then $\abs{\Irr(\mcC_\chi)}=1$, for any non-trivial character $\chi$ of $H$, by Remark \ref{H in C_e}. In particular, $H$ fixes all the simple objects not in $\Irr(\mcC_e)$, but this contradicts Lemma \ref{l:grading1}. Thus, $|H|=2$ or $3$. Lemma \ref{l:grading1} implies that if $|H|=3$ then $3 \mid |\Irr(\mcC_\chi)|$, for each non-trivial character $\chi$ of $H$. However, this means $\rank \mcC \ge 9$. Therefore, $|H|=2$.

Suppose $H=\{0,h\}$ and $\hat{H} = \{e, \chi\}$. In particular, $\chi(h) =-1$. In view of Lemma \ref{l:grading1}, $\Irr(\mcC_\chi)$ can only have 1 or 2 $H$-orbits.

We first show that $\Irr(\mcC_\chi)$ must have only one $H$-orbit. Assume the contrary. Then we have $\Irr(\mcC_e) = \{\1, h, V_1\}$ and $\Irr(\mcC_\chi) = \{V_2, V_3, V_4, V_5\}$ with $h \cdot V_2 = V_3$ and $h \cdot V_4 = V_5$. Note that $V_1$ cannot be invertible and so $V_1$ is fixed by $H$. In particular, $d_1^2 - n d_1- 2=0$, and $d_1 \ne \pm 1$.  Moreover, $V_1$ is self-dual.

Since $\mcC_e$ is a non-pointed premodular category with exactly two invertible classes of simple objects, \cite{O4} implies that $\dim(V_1)\in\{\sqrt{2},2\}$.  By \cite[Cor. B.12]{DGNO1} if $\dim(V_1)=\sqrt{2}$ then $\mcC_e$ is modular, which implies $\mcC$ is a Deligne product of two modular categories by \cite[Theorem 4.2]{M2}, a contradiction.  If $\dim(V_1)=2$ then $\dim\mcC=12$, so the conditions of Theorem \ref{t:4m} are satisfied.  Applying \cite[Theorem 4.2]{M2} again, we see that the only alternative is that $\mcC$ is a metaplectic category of dimension $12$.  But such categories have rank $5$, a contradiction.

Therefore, $\Irr(\mcC_\chi)$ consists of a single $H$-orbit so that
$\Irr(\mcC_e) = \{1, h, V_1, V_2, V_3\}$ and $\Irr(\mcC_\chi) = \{V_4, V_5\}$ with $h\cdot V_4 =V_5$. Moreover, by Lemma \ref{l:grading3}, $\mcC_e$ is integral so that $\mcC$ is weakly integral. Then one of $V_1, V_2, V_3$ must be fixed by $H$. We may assume $V_2$ is fixed by $H$. Since the stabilizer of $\Irr(\mcC_\chi)$ is trivial, it follows from Lemma \ref{l:grading3} that $d_2 = 2 n_2$, for some positive integer $n_2$. Moreover, $d_1, d_3 \in \mbbZ_+$. From the dimension equation we find $\dim \mcC = 4 d_4^2$, $n_2^2 \mid d_4^2$ and
\begin{equation}\label{eq:dim1}
2 d_4^2 = 2 + d_1^2 + 4 n_2^2 + d_3^2
\end{equation}
The equation modulo 2 implies $(d_1, d_3, d_4^2) \equiv (1,1,0)$ or $(0,0,1) \mod 2$.

We first show that $(d_1, d_3, d_4^2) \equiv (1,1,0) \mod 2$ is not possible. Assume the contrary. Then $d_1, d_3$ are odd and so $d_1^2 \mid d_4^2$. By Lemma \ref{l:grading3}, $V_1$ is not fixed by $H$ ($|H|=2\nmid d_1$). Therefore, $h \cdot V_1 = V_3$ and hence $d_1=d_3$. Now, \eqref{eq:dim1} becomes
\begin{equation}\label{eq:dim2}
d_4^2 = 1 + d_1^2 + 2 n_2^2
\end{equation}
and hence $d_4^2$ is even. Moreover, $n_2^2$ and $d_1^2$ are relatively prime, and $d_1^2 n_2^2 = \lcm(d_1^2, n_2^2)$. If $d_4^2 \equiv 2 \mod 4$, then $n_2^2 \equiv 0 \mod 4$ and hence $4 \mid d_4^2$, a contradiction. Therefore, $4 \mid d_4^2$, and so $n_2$ must be odd. Now, we find
$$ 4 \mid \frac{d_4^2}{d_1^2 n_2^2} = \frac{1}{d_1^2 n_2^2} +  \frac{1}{n_2^2} +  \frac{2}{d_1^2}\,.$$
This forces $n_2^2 = d_1^2=1$ and hence $d_2^2=d_4^2=4$ and $\dimC=16$. Thus, $G=G(\mcC)$ has order $4$ and the homogeneous component $\mcC_\psi$ has $\dim \mcC_\psi = 4$, for all $\psi \in \hat{G}$. In particular, $|\Irr(\mcC_\psi)|=1$ for any non-trivial character $\psi$ of $\hat{G}$. Therefore, $V_4$ must be fixed by $G$, contradicting $h\cdot V_4=V_5$.

Now, we have $(d_1, d_3, d_4^2) \equiv (0,0,1) \mod 2$, we proceed to show that $\mcC$ is strictly weakly integral of $\dimC=28$. Let $n_i = d_i/2$ for $i =1,2, 3$. Then $n_i^2 \mid d_4^2$ and hence $n_i$ is odd for $i=1,2,3$. \eqref{eq:dim1} becomes
\begin{equation}\label{eq:dim3}
d_4^2 = 1 + 2 n_1^2 + 2 n_2^2 + 2 n_3^2 \,.
\end{equation}
Now, this equation implies $d_4^2 \equiv 7 \mod 8$, since $1$ is the only odd square modulo $8$. Therefore, $d_4 \not\in \mbbZ$. Let $l = \lcm(n_1^2, n_2^2, n_3^2)$. Then $l$ is the square of an odd integer and hence $l \equiv 1 \mod 8$. Since $n_i^2 \mid d_4^2$, $m =\frac{d_4^2}{l} \equiv 7 \mod 8$. Therefore,
$$ 7 \le m =  \frac{d_4^2}{l} = \frac{1}{l} + 2 \frac{n_1^2}{l} + 2 \frac{n_3^2}{l} + 2 \frac{n_3^2}{l} \le 7\,. $$
This forces $n_1^2=n_2^2=n_3^2 = l =1$. Hence $d_1=d_2=d_3=2$, $d_4^2=7$ and $\dimC=28$.
\end{proof}

From this we obtain:
\begin{theorem}
The only strictly weakly integral rank 7 categories are metaplectic categories.
\end{theorem}
\begin{proof}
Assume that $\mcC$ is strictly weakly integral of rank 7.
By \cite[Theorem 3.10]{GN} we have that $2\mid |G(\mcC)|$. Moreover, by Lemma \ref{1-non-integral-object-ising}, it follows that $|G(\mcC)|\leq 5$. So there are two cases to consider: $G(\mcC)=U(\mcC)\cong \mbZ_2$ or $|G(\mcC)|=4$.

First suppose $\mcD:=\mcC_{pt}$ has rank 2. Clearly $\mcD$ is not modular, as \cite{M2} implies that $\mcC$ can have no modular subcategories.  In particular $\mcD$ is premodular, and hence symmetric. If $\mcD$ is Tannakian, {i.e.} $\mcD\cong\Rep(\Z_2)$ then Prop. \ref{p:bosonic_obj} implies that $\dimC=28$ and we are done by Theorem \ref{t:4m}. Otherwise $\mcD\cong\sVec$ and we have $\mcC_{ad}^\prime=\mcD$ hence $\mcC_{ad}$ is slightly degenerate (\cite{ENO2}).  In particular $\mcC_{ad}$ must have even rank by \cite[Cor. 2.7]{ENO2}. It follow from \cite{B2} that $\mcC_{ad}$ has dimension 10, so $\dimC=20$.  This is impossible in rank $7$.

Now if $|G(\mcC)|=4$, consider the possible categories $\mcC_{ad}=\mcC_e$ corresponding to the universal grading.  Clearly $\mcC_{ad}$ has rank at least 2 and at most 4 (by the pigeonhole principle) and even dimension.  The classification of low-rank (integral) ribbon categories \cite{O1,O2,B2} gives possible simple dimensions in $\mcC_{ad}$ as: $\{1,1\}$,  $\{1,1,2\}$, $\{1,1,1,1\}$ and $\{1,1,2,2\}$, since the pointed subcateory of $\mcC_{ad}$ must have even dimension.  If $\mcC_{ad}$ is pointed with rank $2$ then one would have a rank $7$ category of dimension $8$, which is absurd.  In the second case the remaining three components must have dimension $6$, two of which contain one simple object of dimension $\sqrt{6}$.  The other component must contain the remaining two invertible objects, which is impossible.
If $\mcC_{ad}$ is pointed of rank $4$, the remaining three objects must lie in distinct components and each have dimension $2$.  Such a category is not strictly weakly integral.  In the last case we again have three remaining simple objects of equal dimension, since they reside in distinct components.  But two of them are invertible, which is impossible.
\end{proof}
It remains to consider integral modular categories of rank $7$. For this we employ the methods developed on subsection \ref{s:support weakly integral}.
\begin{lemma}\label{l:bounded_p}
Let $\mcC$ be a weakly integral modular category  of rank $\le 7$. If $p$ is an odd prime factor of $\FSexp(\mcC)$ then $p \le 7$ and $v_p(\FSexp(\mcC))=1$. If rank $\mcC =6$, $2 < p \mid \FSexp(\mcC)$ implies $p \le 5$.
\end{lemma}
\begin{proof}
Suppose $\mcC$ is a weakly integral modular category of rank $\le 7$ and $p$ is the largest odd prime factor of $\FSexp(\mcC)$. Let $\s$ be a $p$-automorphism. Then, by Lemma \ref{l:tech1}, $p \le 13$. For any prime $p=5,7,11,13$, $v_p(\FSexp(\mcC))\le 1$ otherwise $\hs$ admits a support cycle of length $(p^2 - p)/2 \geq 10$ by Lemma \ref{l:tech1}. This certainly won't happen in ranks $\le 7$.

 If $p=13$, then $\hs=(1,2,3,4,5,6)$ by Lemma \ref{l:tech2}. Since the centralizer of $(1,2,3,4,5,6)$ in $S_7$ is $\langle(1,2,3,4,5,6)\rangle$, we get $\Gal(\mcC) = \langle(1,2,3,4,5,6)\rangle$. However, this contradicts Theorem \ref{t:gal}. Therefore, $p<13$.

 If $p=11$, then $(1,2,3,4,5)$ is the unique $p$-support cycle of $\s$. Thus,
 $$
 \hs =(1,2,3,4,5)    \text{ or } (1,2,3,4,5) (0,6).
 $$
 By Theorem \ref{t:gal}, $\Gal(\mcC) = \langle \hs \rangle$ implies rank $\mcC = 11$ or $9$. Therefore, $\langle \hs \rangle \subsetneq \Gal(\mcC)$ and $\dim \mcC = \frac{11 d_1^4}{4}$. Since the centralizer of $\langle \hs \rangle$ is $\langle(1,2,3,4,5), (0,6)\rangle$, $\Gal(\mcC) = \langle(1,2,3,4,5), (0,6)\rangle$, and we have
 $$
 \frac{11 d_1^4}{4} = 2+ 5 d_1^2\,.
 $$
 However, the equation has no integral solution for $d_1^2$. Therefore, $p \le 7$.

 If rank $\mcC =6$ and $p=7$, then $\s$ has a unique support cycle $(1,2,3)$ and $ \langle \hs \rangle \subsetneq \Gal(\mcC)$ by Theorem \ref{t:gal}. Thus, $\Gal(\mcC) = \langle (1,2,3), (0,5)\rangle$ or $\langle (1,2,3), (4,5)\rangle$. The second case implies $\mcC$ is integral and
 $$
  \frac{7 d_1^4}{4} = 3 d_1^2 +  2 d_4^2 + 1\,.
 $$
 by Theorem \ref{t:gal}(i).
 The right hand side of this equation is an
 integer and we deduce that $2\mid d_{1}$. Reduction modulo $2$ now reveals the contradictory expression: $0 \equiv 1 \mod 2$.
 Therefore, $\Gal(\mcC) = \langle (1,2,3), (0,5)\rangle$ and
  $$
 \frac{7 d_1^4}{4} = 3 d_1^2 +   d_4^2 + 2\,.
 $$
 Next observe, that the Galois group moves $0$ and hence $\mcC$ is strictly
 weakly integral. In this case, \cite{GN} implies that $\mcC$ is faithfully
 graded by an elementary abelian 2-group with the trivial component being given
 by the integral subcategory. So by dimension count we see that either $d_{1}$
 or $d_{4}$ is and integer and the other dimension is not. Of course, if
 $d_{4}\notin\mbbZ$, then it corresponds to a unique object of non-integral dimension and
 so \lemmaref{1-non-integral-object-ising} implies that $\mcC$ is an Ising modular category, an
 impossibility. Therefore, $d_{4}\in\mbbZ$ and $d_{1}\notin\mbbZ$.
 So $\frac{7 d_1^4}{8} = 3 d_1^2$ but this implies $d_1^2=0$ or $\frac{24}{7}$, a contradiction.
 \end{proof}

 \begin{lemma}\label{l:3+3}
  Suppose $\mcC$ is a weakly integral modular category of rank 7 such that $d_1=d_2=d_3\leq d_4= d_5=d_6$. Then, one of the following statements holds:
  \begin{enumerate}
    \item $(d_1, d_4)=(1,1)$ and $\dimC=7$.
    \item $(d_1, d_4)=(1,\sqrt{2})$ and $\dimC=10$.
    \item $(d_1, d_4)=(1,2)$ and $\dimC=16$.
  \end{enumerate}
\end{lemma}
  \begin{proof}
   The assumption implies the equality
    $$
 \dimC = 1 + 3 d_1^2+3 d_4^2\,.
  $$
   Thus $d_1^2, d_4^2$ are relatively prime and so $d_1^2 d_4^2 \mid \dimC$. Then
  \begin{equation}\label{eq:int1}
    d_1^2 \mid \frac{\dimC}{d_4^2} =\frac{1}{d_4^2} + 3 \frac{d_1^2}{d_4^2} +3 \in \mbbZ
  \end{equation} implies $\frac{\dimC}{d_4^2}\leq 7$.  In particular $d_1^2\leq 7$ and $d_4^2\mid (1+3d_1^2)$.

  Now if $\mcC$ is strictly weakly integral, we must have $d_1=1$ since $|U(\mcC)|\neq 1$.  In this case $d_4^2$ divides $4$ so that non-integrality implies $d_4=\sqrt{2}$.

  Otherwise, $d_1\in\mbbZ$ implies $d_1=1$ or $d_1=2$.  If $d_1=1$ and $d_4\in\mbbZ$ then $d_4\in\{1,2\}$ since $d_4^2\mid 4$.  If $d_1=2$, we have $d_4^2\mid 13$, which is impossible since $d_4\geq d_1$.

  Thus the three possibilities are as in the statement.

  \end{proof}

 \begin{remark}\label{remark-dim-max}
   The proof of this lemma reveals that if $\mcC$ is a non-pointed weakly
   integral modular category and $d_{\max}^2 = \max_i{d_i^2}$, then $\frac{\dimC}{d_{\max}^2}$ is a positive integer strictly less than $\rank \mcC$.
 \end{remark}
 \begin{prop} \label{p:7}
   If $\mcC$ is a weakly integral modular category of rank $7$ and $7 \mid \FSexp(\mcC)$, then $\mcC$ is either pointed or $\dim \mcC =28$.
 \end{prop}
 \begin{proof} Assume $\mcC$ is not pointed. We first show that $\Gal(\mcC)$ does not contain any permutation of type $(1,3,3)$ or $(1,6)$ with $0$ fixed. Assume the contrary. Then $d_1=d_2=d_3$, $d_4=d_5=d_6$ and we have
$$
\dimC = 1 + 3 d_1^2+3 d_4^2.
$$
By Lemma \ref{l:3+3}, $\mcC$ must be pointed, a contradiction.

Let $p =7$, $\s$ the $p$-automorphism of $\mcC$ and $C_1$  a $p$-support cycle of $\s$. Then, in view of Lemma \ref{l:tech2}, $\ord(C_1)=3$, say $C_1=(1,2,3)$. Since all the 7-support cycles must have length $\ge 3$, $C_1$ is the unique 7-support cycle of $\s$ and all other disjoint cycles are of length less than 3. It follows from Theorem \ref{t:gal} that $\dimC = \frac{7 d_1^2}{4}$ and $2 \mid d_1^2$. Moreover,
$$
\hs \in\{ (1,2,3),\, (1,2,3)(4,5),\, (0,6)(1,2,3)(4,5) \text{ or } (0,6)(1,2,3)\}
$$
 up to renumbering the non-zero labels.

If $\hs$ were one of the first three cases, then $\langle \hs \rangle \ne \Gal(\mcC)$ otherwise it will contradicts Theorem \ref{t:gal}. Since $\Gal(\mcC)$ is an abelian subgroup of $S_7$ containing $\hs$, there exists an element of the form $C_1^i(4,5), (0,6)C_1^i(4,5)$ in $\Gal(\mcC)$, for some $i=1,2$. In particular, we have the equation
 \begin{equation}\label{eq:3+2:1}
 \frac{7 d_1^4}{4} = 1+ 3 d_1^2+ 2 d_4^2+ d_6^2,
 \end{equation}
  with $2 \mid d_1^2$.

 We claim that if the dimensions the simple objects of $\mcC$ satisfy \eqref{eq:3+2:1}, then $\mcC$ is a prime modular category of dimension 28. We first observe that the equation \eqref{eq:3+2:1} modulo 2 implies that the parities of $\frac{d^2_1}{2}$ and $d_6^2$ are opposite. Since $d_6^2 \mid \frac{7 d_1^4}{4}$, $d_6^2$ must be odd and $d_1^2/2$ is even. Thus,
  \begin{equation} \label{eq:3+2}
  28 n_1^2= 1+ 12 n_1+ 2 d_4^2+ d_6^2\, ,
  \end{equation}
 where $d_1^2 = 4n_1$ for some positive integer $n_1$. If $d_6^2 \equiv 3 \mod 4$, then $d_6 \not\in \mbbZ$ and so there exists $\t \in \Gal(\mcC)$ which admits a transposition of the form $(0,j)$, and in particular, $d_j=1$. Thus, $j$ can only be $4$ or $5$ but this does not balance the equation \eqref{eq:3+2} modulo $4$. Therefore, $d_6^2 \equiv 1 \mod 4$ and hence $d_4^2$ is odd.

Since $d_{\max}^2 \mid \dimC = 28 n_1^2$ and $\frac{\dimC}{d_{\max}^2} < 7$,
$d_{\max}^2 \ne d_1^2$ and so $d_{\max}^2 = d_4^2$ or $d_6^2$. In particular, $d_{\max}^2$ is odd.
Thus, $\frac{\dimC}{d_{\max}^2} =4$ or $d_{\max}^2=7n_1^2 \equiv 3 \mod 4$. Hence, $d_{\max}^2= d_4^2$ and \eqref{eq:3+2} becomes
$$
14 n_1^2= 1+ 12 n_1+ d_6^2\,.
$$
In particular, $d_6^2$ and $n_1$ are relatively prime. Since $d_6^2 \mid 28 n_1^2$ and $d_6^2 \equiv 1\mod 4$, $d_6^2 = 1$ and so
  $$
  7n_1^2 - 6 n_1 -1 = 0\,.
  $$
This equation forces $n_1=1$ and hence $(d_1^2, d_4^2, d_6^2) = (4,7,1)$. Therefore, $\mcC$ is a prime modular category of dimension $28$, and this proves the claim.

As a consequence of the preceding claim, $\hs = (0,6)(1,2,3)$. To complete the proof, it suffices to show that $\langle\hs\rangle = \Gal(\mcC)$. If not, then $\Gal(\mcC)$ contains $(4,5)$ and so the dimensions of $\mcC$ satisfy \eqref{eq:3+2:1} again. This implies $\mcC$ is a prime modular category of dimension 28 but then $\Gal(\mcC) = \langle\hs\rangle$, which yields to a contradiction.
\end{proof}

\begin{theorem} If $\mcC$ is an integral modular category of rank $7$, then $\mcC$ is pointed.
\end{theorem}
\begin{proof} Let $\mcC$ be an integral modular category of rank $7$, and assume it is not pointed. By Lemma \ref{l:bounded_p}, the prime factors of $\FSexp(\mcC)$ can only be $2,3,5, 7$. By the Cauchy Theorem \cite{BNRW}, the prime factors of $\dimC$ can only be $2,3,5, 7$. In view of Proposition \ref{p:7}, it suffices to prove that $7 \mid \dimC$. Equivalently, it enough to show none of $2,3$ or $5$ is a prime factor of $\dimC$. In view of Proposition \ref{p:bosonic_obj}, it is not possible that only two of the these primes are factors of $\dimC$. It suffices to show that $\dimC=2^a 3^b 5^c$  with $abc \ge 1$ is not possible.

Suppose $\dimC=2^a 3^b 5^c$  with $abc \ge 1$. Let $\s$ be a $5$-automorphism of $\mcC$, and $C_1$ a $5$-support cycle of $\s$. We first show that length of $C_1$ must be 2. If not, then $\ord(C_1)=4$ and so that:
\begin{equation}\label{eq:4+0}
\dim \mcC = 1+ 4 d_1^2 + d_5^2+d_6^2
\end{equation}
where we have ordered the simple objects so that $C_1=(1,2,3,4)$. This equation modulo 2 implies
$$
0 \equiv 1+ d_5^2+d_6^2 \mod 2\,.
$$
Without loss of generality, we may assume $d_5$ is odd and $d_6$ is even. In particular, $4 \mid \dimC$ as $d_6^2 \mid \dimC$. However, we then find
$0 \equiv 2    \mod 4 $, a contradiction. Therefore, the dimensions of $\mcC$ do not satisfy \eqref{eq:4+0}, and so $\Gal(\mcC)$ does not contain any permutation which admits a disjoint cycle of length $\ge 4$. In particular, $\ord(C_1) < 4$.

Now, we may assume $C_1 =(1,2)$ and proceed to show that this is a unique $5$-support cycle of $\s$. Then $\hs =(1,2)(3,4)$ or  $\hs =(1,2)(3,4)(5,6)$. However, if $\Gal(\mcC)$ contains any permutation of type $(1,2,2,2)$ with $0$ fixed,  the we have:
 \begin{equation}\label{eq:2+2+2}
\dimC = 1+ 2 d_1^2 + 2 d_3^2+2 d_5^2 \equiv 1 \mod 2\,.
\end{equation}
This is not possible since such a category is pointed by \cite{BR1}. Moreover, it implies that $\Gal(\mcC)$ does not contain any permutation of type $(1,2,2,2)$ with $0$ fixed. In particular, $\hs = (1,2)(3,4)$.

If $(3,4)$ is also a 5-support cycle of $\hs$ then, by Lemma \ref{l:tech3}, the dimensions $d_i$ of $\mcC$ satisfy \eqref{eq:4+0} or
\begin{equation}\label{eq:*2+2}
1+ 2 d_1^2 + 2d_4^2  + d_5^2+d_6^2 = \frac{1}{4} (d_1^4+d_4^4+\e d_1^2 d_4^2), \text{ with } 2 \mid d_1, d_2 \text{ and } \e=0,1, -1.
\end{equation}
 We have shown that the dimensions $d_i$ of $\mcC$ do not satisfy \eqref{eq:4+0}. By considering \eqref{eq:*2+2} modulo 2, we may assume that $d_5$ is odd and $d_6$ is even. Modulo 4,  \eqref{eq:*2+2} then becomes
$$
0 \equiv  1+ d_5^2 \mod 4,
$$
but this is impossible as $d_5^2\equiv 1\mod 4$. Therefore, $(3,4)$ is not a 5-support cycle if $\hs = (1,2)(3,4)$. In particular, $(1,2)$ is the unique $5$-support of $\s$. By Theorem \ref{t:gal}, we find $\dimC = \frac{5 d_1^4}{4}$ and $d_1 = 2  n_1$, for some positive integer $n_1$.

Suppose there exists a permutation $\t \in \Gal(\mcC)$ which admits a cycle of length $\ge 2$ and disjoint from $(1,2)$, say $(3,4,\dots)$. Then the $d_i$ must satisfy:
\begin{equation}\label{eq:2+2}
  20 n_1^4 = 1+ 8 n_1^2 + 2 d_4^2+ d_5^2 + d_6^2\,.
\end{equation}
Then $d_5$  and $d_6$ must have opposite parities, otherwise the left hand side of \eqref{eq:2+2} would be congruent to 1 modulo 2. We may assume $d_5$ is odd and $d_6 = 2 n_6$ for some positive integer $n_6$. Now, \eqref{eq:2+2} modulo 4 yields
$$
0 \equiv 2  + 2 d_4^2 \mod 4\,.
$$
This forces $d_4$ to be odd, and we have
$$
 \dimC \equiv 4 +  4 n_6^2 \mod 8\,.
$$
Thus, $n_6$ must be odd, $8 \mid \dimC$, and so $n_1$ is also even. Let $d_1 = 4 m_1$ for some positive integer $m_1$.  Now, $\frac{\dimC}{d_j^2} = \frac{2^6 \cdot 5 m_1^4}{d_j^2} > 7$ for $j=1,\dots,6$. But this contradicts that $\frac{\dimC}{d_{\max}^2} < 7$, see Remark \ref{remark-dim-max}. Therefore, no permutation $\t \in \Gal(\mcC)$  admits a non-trivial cycle disjoint from $(1,2)$. Hence, $\Gal(\mcC) = \langle \hs \rangle$ and $\hs=(1,2)$, which leads to a contradiction of Theorem \ref{t:gal}.
\end{proof}
\bibliographystyle{plain}

\end{document}